\documentclass{amsart}
\usepackage{amsfonts}
\usepackage{amsmath}
\usepackage{amssymb}
\usepackage{amsthm}

\makeatletter 
\def\@tocline#1#2#3#4#5#6#7{\relax
  \ifnum #1>\c@tocdepth 
  \else
    \par \addpenalty\@secpenalty\addvspace{#2}
    \begingroup \hyphenpenalty\@M
    \@ifempty{#4}{
      \@tempdima\csname r@tocindent\number#1\endcsname\relax
    }{
      \@tempdima#4\relax
    }
    \parindent\z@ \leftskip#3\relax \advance\leftskip\@tempdima\relax
    \rightskip\@pnumwidth plus4em \parfillskip-\@pnumwidth
    #5\leavevmode\hskip-\@tempdima
      \ifcase #1
       \or\or \hskip 1em \or \hskip 2em \else \hskip 3em \fi
      #6\nobreak\relax
    \dotfill\hbox to\@pnumwidth{\@tocpagenum{#7}}\par
    \nobreak
    \endgroup
  \fi}
\makeatother

\DeclareMathAlphabet{\mathpzc}{OT1}{pzc}{m}{it}
\newcommand{\hh}{\mathbb{H}}

\newcommand{\B}{{\mathbb B}}
\newcommand{\C}{{\mathbb C}}

\renewcommand{\H}{{\mathbb H}}
\renewcommand{\S}{{\mathbb S}}
\newcommand{\N}{{\mathbb N}}

\newcommand{\Z}{{\mathbb Z}}

\newcommand{\R}{{\mathbb R}}

\newcommand{\tensor}{\otimes}

\let\Realpart\Re
\renewcommand{\Re}{\mathop{\Realpart e}}

\let\Impart\Im
\renewcommand{\Im}{\mathop{\Impart m}}

\newcommand{\HC}{\H\tensor\C}

\newtheorem{theorem}{Theorem}[section]
\newtheorem*{theorem*}{Theorem}
\newtheorem{proposition}[theorem]{Proposition}
\newtheorem{corollary}[theorem]{Corollary}
\newtheorem{lemma}[theorem]{Lemma}
\theoremstyle{definition}
\newtheorem{definition}[theorem]{Definition}

\newtheorem{remark}[theorem]{Remark}

\newtheorem*{warning}{Warning}

\title{The Harmonicity of slice regular functions} 

\author[Cinzia Bisi]{Cinzia Bisi}
\address{Cinzia Bisi : Dipartimento di Matematica e Informatica, Universit\`a di Ferrara, via Machiavelli 30, I-44121 Ferrara, Italy}
\email{bsicnz@unife.it \newline
ORCID: 0000-0002-4973-1053}
\author[J\"org Winkelmann]{J\"org Winkelmann}
\address{J\"org Winkelmann : Lehrstuhl Analysis II, Fakult\"at f\"ur Mathematik, Ruhr-Universit\"at Bochum, 44780 Bochum, Germany}
\email{Joerg.Winkelmann@ruhr-uni-bochum.de\newline
  ORCID: 0000-0002-1781-5842
}

\begin{document}

\begin{abstract}
 In this article we investigate harmonicity, Laplacians, mean value theorems
  and related topics in the context of quaternionic analysis. 
We observe that a Mean Value Formula for slice regular functions holds true and it is a consequence of the well known Representation Formula for slice regular functions over $\mathbb{H}$. Motivated by this observation, we have constructed three order-two differential operators in the kernel of which slice regular functions are, answering positively to the question:
 is a slice regular function over $\mathbb{H}$ (analogous to an holomorphic function over $\mathbb{C}$)  ''harmonic" in some sense, i.e. is it in the kernel of some order-two differential operator over $\mathbb{H}$? \\
 Finally, some applications are deduced, such as a Poisson Formula for slice regular functions over $\mathbb{H}$ and a Jensen's Formula for semi-regular ones.
\end{abstract}

\maketitle

\tableofcontents

\section{Introduction}
In \cite{gentilistruppa1} and \cite{gentilistruppa}, Gentili and Struppa gave the following definition of {\it slice regular function} over the quaternions:
\begin{definition}
  Let $\Omega$ be a domain in $\H$. A real differentiable function $f \colon \Omega \to \H $ is said to be slice regular if, $\forall \,\, I \in \S =\{ q \in \H , \, \Re(q)=0 \, : \, |q| =1 \}$,
  its restriction $f_{I}$ to the complex line 
  $\C_I= \mathbb{R} + \mathbb{R} I$ passing through the origin and containing $1$ and $I$ is holomorphic on $\Omega \cap \C_I$,
  which is equivalent to require that, $\forall \, I \, \in \S ,$
\begin{equation}\label{defdbar}
\overline{\partial}_I f (x+yI) := \frac{1}{2} ( \frac{\partial}{\partial x} + I \frac{\partial}{\partial y})f_I (x+yI)=0
\end{equation}
on $\Omega \cap \C_I.$
\end{definition}
Later the notion of ``slice regularity'' was generalized to algebras
other than $\H$ (\cite{CSS09,ghiloniperotti}).

For simplicity, sometimes slice regukar functions are simply
called ``regular functions''.

  
  Let $D\subset\mathbb{C}$ be any symmetric set
  with respect to the real axis.
  A function $F=F_{1}+ F_{2}\imath:D\rightarrow \H\tensor{\mathbb{C}}$
  such that $F(\overline z)=\overline{F(z)}$ is said to be a
  \textit{stem function}.

  Let $\Omega_D=\{\alpha+\beta I:\alpha,\beta\in\R, I\in\S,
  \alpha+\beta i\in D\}$.
  
  A function $f:\Omega_{D}\rightarrow \H$ is said to be a \textit{(left) slice
function} if it is induced by a stem function $F=F_{1}+F_2\imath $ defined on $D$ 
in the following way: for any $\alpha+I\beta\in\Omega_{D}$,
\begin{equation*}
f(\alpha+I\beta)=F_{1}(\alpha+i\beta)+I F_{2}(\alpha+i\beta).
\end{equation*}

\noindent
If a stem function $F$ induces the slice function $f$, we will write
$f=\mathcal{I}(F)$. 

\begin{proposition}
  Let $D$ be a symmetric domain in $\C$ which intersects $\R$
  and let $\Omega_D\subset\H$ be defined as above.

  Then a {\em slice function} $f:\Omega_D\to \H$ is slice
  regular if and only
  if its stem function $F:D\to \H\tensor\C$ is holomorphic.
\end{proposition}
(See Proposition~8 of \cite{ghiloniperotti}.)

These notions have been studied a lot in the last years:
see, for example, the many results for slice regular functions from the unit ball of $\H $ to itself: \cite{bisigentili}, \cite{bisigentilitrends}, \cite{bs12}, \cite{bisistoppato}, \cite{bs17} and for entire slice regular functions \cite{BW}.\\

Classically, mean value theorems are closely related to harmonicity.
We investigate mean value properties for quaternionic functions.
We prove (Proposition~\ref{gen-mean}) that a slice regular function $f$
fulfills
\[
    f(a+bI)
    =\frac{1}{2\pi}
  \int_\S  \int_0^{2\pi}
  (1-IJ)f(a+bJ+re^{J\theta})
  d\theta
  d\mu(J), \ \forall a,b\in\R, I\in\S
  \]
  where $\mu$ is a probability measure on $\S$ which is
  invariant under the involution $J\mapsto -J$.

  Conversely, we show that every continuous function $f:\H\to\H$ with
  this mean value property must be the sum of a regular and an
  anti-regular function (Theorem~\ref{char-harm}).

  We also show that for any point $p\in\H$ and every $3$-sphere $S$
  containing $p$ in the interior, there exists a $\H$-valued measure
  on $S$ such that $f(p)=\int_S f(q)d\mu(q)$ for every slice regular function
  $f$ (Theorem~\ref{H-measure}).

Over the field of complex numbers, the mean value property is equivalent to harmonicity.
Therefore it is natural ask ourselves if slice regular functions were in the kernel of some order-two differential operator over $\H $: in section 8 we answer positively to this question
constructing three order-two differential operators in the kernel of which slice regular functions are.

The first one is $\Delta_*$, introduced in
Definition~\ref{def-laplace-star}.
For slice functions it is defined everywhere, for other functions
only outside $\R$.
On each slice $\C_I$ the operator $\Delta_*$ acts as the complex
Laplacian if we identify $\C_I\simeq\C$ and $\H=\C_I\oplus J\C_I
\simeq\C^2$ (with $J$ being an imaginary unit orthogonal to $I$).
If $f$ is a slice function, then $\Delta_*(f)$ is again a slice function
and $\Delta_*(f)=\mathcal{I}(\Delta F)$
for $f=\mathcal{I}(F)$.

The second order-two differential operator is $\Delta'$:
\[
(\Delta'f)(q)=\left(\Delta_* \int_{\S}R_wfd\mu(w)\right) (q)
\]
Here $R_w$ is an averaging operator which we define based on rotations,
cf.~section \ref{sect-rot}.
We observe that $(\Delta' f)(q)=\frac{1}{2}\Delta_* (Tr(f))(q)=
\frac{1}{2}\Delta_* (f+f^c)(q)$. For the definition of $f^c$ see Definition 2.10. \\

The third order-two differential operator is $\Delta''$:
$$
(\Delta'' f)(q)=(\Delta_* N(f))(q)=(\Delta_* (f \cdot f^c)) (q)
$$

On one side $\Delta_*$ and $\Delta '$ are $\mathbb{R}-$linear operator, on the other hand $\Delta''$ is not a linear operator: but for $\Delta ''$ a sort of 
Leibnitz rule for $(f * g)$ holds true
(proposition~\ref{product-rule}).
For the definition of slice product, denoted with $*$, see Definition 2.7.\\

Our main results on  $\Delta_*$ and $\Delta'$ are the following :
\begin{theorem}[Theorem~\ref{harm-star}]
Let $f:\Omega_D\to\H$ be a $C^2$ slice function.
Assume that $D$ is simply-connected.

$\Delta_*f$ is vanishing identically if and only if $f$
can be written as a sum of a regular function $g$ and an
anti-regular function $h$.
\end{theorem}
For the definition of anti-regular function see Remark 2.5.

\begin{proposition}[Proposition~\ref{real-part}]
Let $h:\H \to\mathbb{R}$ be a 
slice function with $\Delta'h=0$.
Then there exists a slice-preserving regular function $f$ such that $h=\Re  \, (f)$.
\end{proposition}
\begin{proposition}[Proposition~\ref{isol-zero}]
Let $u \colon \H  \to \mathbb{R}$ be 
a $C^2$-function such that $\Delta' u=0$ outside $\R$. 

Then $u$ admits no isolated zero in any real point $a\in\R$.
\end{proposition}

Finally, we provide a Jensen's formula.

\begin{proposition}[Jensen's Formula; Proposition \ref{jensen}]
Let $\Omega=\Omega_D$ be a circular
domain of $\H $ and 
let $f:\Omega\rightarrow\H \cup\{ \infty \}$ be a
semi-regular function.
Let $\rho>0$ be such that the ball, centered in 0 and of radius $\rho,$ 
$\overline{\mathbb{B}_{\rho}}\subset\Omega$, $f(0)\neq 0,\infty$ 
and such that $f(y)\neq 0, \infty$, for any $y\in\partial \mathbb{B}_{\rho}$.
Let $\mu$ be a probability measure on $\S$.

Then: \\
\begin{equation}
\begin{array}{ccc}
\log|f(0)| 
&\le 
\displaystyle\frac{1}{2\pi \mu(\S)}\displaystyle \int_0^{2\pi} \int_{\S }\log|f(\rho\cos\theta +\rho\sin\theta I)|d\mu(I)d\theta \, + \\
&\\
& \displaystyle-\sum_{|p_k|<\rho} m_k\log \frac{\rho}{|p_k|}
\quad 
 \\
\end{array}
\end{equation}
for $div(f)=\sum m_k\{p_k\}$.
\end{proposition}

\vskip0.3cm
We hope that this paper can provide new ideas for studying slice regular functions and their "harmonic properties'' on slice regular
quaternionic manifolds recently introduced by Bisi-Gentili in \cite{BG} and Angella-Bisi in \cite{AnB}.

\section{Prerequisites about quaternionic functions}

In this section we will overview the main notions and results needed
for our aims. 
First of all, let us denote by $\mathbb	{H}$ the real algebra
of quaternions. An element $x\in\H $ is usually written as 
$x=x_{0}+ix_{1}+jx_{2}+kx_{3}$, where $i^{2}=j^{2}=k^{2}=-1$ and $ijk=-1$.
Given a quaternion $x$ we introduce a conjugation in $\H $ (the usual one),
as $x^{c}=x_{0}-ix_{1}-jx_{2}-kx_{3}$; with this conjugation we define the real
part of $x$ as $\Re(x):=(x+x^{c})/2$ and the imaginary part as $\Im(x):=(x-x^{c})/2$.
With the notion of conjugation just defined
we can write the euclidean square norm of a
quaternion $x$ as $|x|^{2}=xx^{c}$. 
The subalgebra of real numbers will be identified, of course, 
with the set
$\mathbb{R}:=\{x\in\H \,\mid\,\Im(x)=0\}.$

Now, if $x\in\H \setminus \mathbb{R}$ is such that $\Re (x)=0$, 
then the imaginary part of $x$ is such that 
$(\Im(x)/|\Im(x)|)^{2}=-1$. More precisely, any imaginary quaternion 
$I=ix_{1}+jx_{2}+kx_{3}$, such 
that $x_{1}^{2}+x_{2}^{2}+x_{3}^{2}=1$ is an imaginary unit. The set of imaginary units 
is then a real $2-$sphere and it will be conveniently denoted as follows:
\begin{equation*}
\S :=\{x\in\H \,\mid\, x^{2}=-1\}=\{x\in\H \,\mid\, 
\Re(x)=0, \, |x|=1\}.
\end{equation*}

With the previous notation, any $x\in\H $ can be written as $x=\alpha+I\beta$,
where $\alpha,\beta\in\mathbb{R}$ and $I\in\S $. 
Given any $I\in\S $ we will denote the real subspace of $\H $ 
generated by $1$ and $I$ as:
\begin{equation*}
\mathbb{C}_{I}:=\{x\in\H \,\mid\,x=\alpha+I\beta, \alpha,\beta\in\mathbb{R}\}.
\end{equation*}
Sets of the previous kind will be called \textit{slices}.

We denote the $2-$sphere with center
$\alpha\in\mathbb{R}$ and radius $|\beta|$ (passing through $\alpha+I\beta\in\H $), as:
\begin{equation*}
\S _{\alpha+I\beta}:=\{x\in\H \,\mid\,x=\alpha+J\beta, J\in\S \}.
\end{equation*}
Obviously, if $\beta=0$, then $\S _{\alpha}=\{\alpha\}$.

\subsection{Slice functions and regularity}
In this part we will recall the main definitions and features of slice functions. 
The theory of slice functions was introduced in \cite{ghiloniperotti} as a tool to generalize
the theory of quaternionic regular functions defined on particular domains introduced in \cite{gentilistruppa1, gentilistruppa},
to more general domains and to 
alternative $*-$algebras. 

\noindent
The complexification of $\H $ is defined to be the real tensor product 
between $\H $ itself and $\mathbb{C}$: 
\begin{equation*}
\H _{\mathbb{C}}:=\H \otimes_{\mathbb{R}}\mathbb{C}:=\{p+q\imath \,\mid\, p,q\in\H \}.
\end{equation*}
(Here $\imath=1 \tensor i$.)
Note that $\HC $ has a natural structure of an associative
algebra induced by the algebra structures of $\H$ and $\C$.
Explicitly, the product on $\HC$ is given as follows: 
if $p_{1}+ q_{1}\imath , p_{2}+ q_{2}\imath$ belong to $\HC $, 
then,
\begin{equation*}
(p_{1}+ q_{1}\imath)(p_{2}+ q_{2}\imath)=p_{1}p_{2}-q_{1}q_{2}+(p_{1}q_{2}+q_{1}p_{2})\imath.
\end{equation*}
The usual complex conjugation $\overline{p+q\imath }=p-q\imath $ commutes with
the following involution (the quaternionic conjugation)
$(p+q\imath )^{c}=p^{c}+ q^{c}\imath$.

We introduce now the class of subsets of $\H $ where our functions will be defined.
\begin{definition}\label{def-circular}
  Given any set $D\subseteq\mathbb{C}$,
  we define its \textit{circularization} as the 
subset in $\H $ defined as follows:
\begin{equation*}
  \Omega_{D}:=\{\alpha+I\beta\,\mid\,\alpha,\beta\in\R,
  \alpha+i\beta\in D, I\in\S \}.
\end{equation*}
Such subsets of $\H $ are called \textit{circular} sets. 
If $D\subset \mathbb{C}$ is 
an open connected subset such that $D\cap\mathbb{R}\neq\emptyset$, 
then $\Omega_{D}$ 
(which is again open and connected and intersects the real line $\R$)
is called a \textit{slice domain} 
(see \cite{genstostru}). 
\end{definition}

Note that for any subset $D\subset \C$ the circularization
$\Omega_D$ coincides with the circularization $\Omega_{D^s}$ of the
symmetrized domain $D^s=\{z:z\in D\text{ or }
\bar z\in D\}$.

From now on, $\Omega_{D}\subset\H $ will always denote a circular  domain arising as circularization of a symmetric domain $D\subset\C$.

\begin{definition}
Let $D\subset\mathbb{C}$ be any symmetric set with respect to the real axis. A function $F=F_{1}+F_2\imath :D\rightarrow \HC $ such that $F(\overline z)=\overline{F(z)}$ is said to be a \textit{stem function}. 

\noindent
A function $f:\Omega_{D}\rightarrow\H $ is said to be a \textit{(left) slice
function} if it is induced by a stem function $F=F_{1}+F_2\imath $ defined on $D$ 
in the following way: 
\begin{equation}\label{eq-stem}
f(\alpha+I\beta)=F_{1}(\alpha+i\beta)+I F_{2}(\alpha+i\beta).
\end{equation}
for all $\alpha+i\beta\in D$ and
all $I\in \S$.
\noindent
If a stem function $F$ induces the slice function $f$, we will write $f=\mathcal{I}(F)$. 
The set of slice functions defined on a certain circular domain $\Omega_{D}$ will
be denoted by $\mathcal{S}(\Omega_{D})$.
\end{definition}

\begin{lemma}\label{stem-basic}
Let $f:\Omega_D\to\H$ be a function defined on a circular
domain $\Omega_D$. If there exists a function $F=F_1+F_2\imath: D\to\HC$
such that
equation~\eqref{eq-stem} holds, then
$F$ is a stem function, i.e., $F_1(z)=F_1(\bar z)$
and $F_2( z)=-F_2(\bar z)$.
\end{lemma}
\begin{proof}
We observe that for all $z=\alpha+i \beta\in D$ and all $I\in\S$
we have
\[
f(\alpha+I\beta)=F_{1}(\alpha+i\beta)+I F_{2}(\alpha+i\beta)
\]
and
\[
f(\alpha+(-I)(-\beta))=F_{1}(\alpha-i\beta)-I F_{2}(\alpha-i\beta),
\]
implying the statement.
\end{proof}

Examples of (left) slice functions are polynomials and 
functions given by power series in the variable $x\in\H $ with all coefficients on the right, i.e., a power series
\begin{equation*}
\sum_{k=0}^{+ \infty} x^{k}a_{k},\quad \{a_{k}\}\subset\H ,
\end{equation*}
if convergent, defines a slice function.

A function $f:\Omega_D\to\H$ is a {\em slice function}
if and only if it obeys the following {``representation formula''}:
\[
f(x+yJ)=\frac{1-JI}2f(x+yI)+ \frac{1+JI}2f(x-yI)\,\, \forall x,y\in\R,  \,\, \forall
I,J\in\S
\]
(see \cite{genstostru, ghiloniperotti}).

\subsubsection{Regularity}

Let now $D\subset \mathbb{C}$ be an open set and $z=\alpha+i\beta\in D$. Given a stem function $F=F_1+F_2\imath :D\rightarrow \hh_\mathbb{C}$
of class $C^1$, then
\begin{equation*}
 \frac{\partial F}{\partial z},\frac{\partial F}{\partial\bar{z}}:D\rightarrow\hh_\mathbb{C}\simeq\C^4,
\end{equation*}
are defined as usual, i.e., 
\begin{equation*}
 \frac{\partial F}{\partial z}=\frac{1}{2}\left(\frac{\partial F}{\partial \alpha}-\imath \frac{\partial F}{\partial \beta}\right)\quad \mbox{ and }\quad \frac{\partial F}{\partial \bar{z}}=\frac{1}{2}\left(\frac{\partial F}{\partial \alpha}+\imath \frac{\partial F}{\partial \beta}\right).
 \end{equation*}
They are again stem functions.

Let $f$ be a slice function induced by a stem function $F$ (i.e.~$f=\mathcal{I}
(F)$) and let $q\in\H$, $q\in\C_I$, $I\in \S$.

Then
\begin{equation}\label{cullen}
(\partial_If)(q)=\mathcal{I}\left(\frac{\partial F}{\partial z}\right)(q),
\quad
(\bar \partial_If)(q)=
\mathcal{I}\left(\frac{\partial F}{\partial \overline{z}}\right)(q).
\end{equation}

These derivatives are 
also called ``Cullen derivatives''.

We are now in the position to define slice regular functions (see Definition 8 in \cite{ghiloniperotti}).
\begin{definition}
  Let $\Omega_{D}$ be a circular open set.
  A slice function $f=\mathcal{I}(F)\in\mathcal{S}(\Omega_{D})$
  is \textit{(left) regular} if its stem function $F$ is holomorphic. 
The set of regular functions will be denoted by
\begin{equation*}
\mathcal{SR}(\Omega_{D}):=\{f\in\mathcal{S}(\Omega_{D})\,|\,f=\mathcal{I}(F), \,\,\, F:D\rightarrow \HC \mbox{ holomorphic}\}.
\end{equation*}
\end{definition}

The set of regular functions is a real vector space and a right $\H $-module.
In the case in which $\Omega_{D}$ is a slice domain, the definition of regularity is equivalent to the one given in \cite{genstostru}.

\begin{remark}
A function $f=\mathcal{I}(F)\in\mathcal{S}^{1}(\Omega_{D})$ is 
called \textit{(left) anti-regular} if its stem function $F$ is anti-holomorphic.
\end{remark}

We recall a key lemma of this theory that will be useful later on, \cite{genstostru}.
\begin{lemma}[Splitting]\label{splitting}
  Let $f$ be a regular function defined on an open set $\Omega$ of $\mathbb{H}.$ Then, for any $I \in \mathbb{S}$ and for any $J \in \mathbb{S}$ with $J \perp I,$ there exist two holomorphic functions
  $g_I,h_I \colon \Omega \cap \mathbb{C}_I \to \mathbb{C}_I$ such that, $\forall z = x + y I$, it is :
$$
f_I (z)=g_I(z)+h_I(z)J
$$
where $f_I$ is the restriction of $f$ to $\mathbb{C}_I.$
\end{lemma}

\subsubsection{Product of slice functions and their zero set}

In general, the pointwise product of slice functions is not a slice 
function.
However there is some product called ``slice product'' which 
does turn slice functions into slice functions.

The following notion is of great importance in the theory.
For the following basic facts on this ``slice product''
see 
\cite{ghiloniperotti}
and \cite{genstostru}. 

\begin{definition}
Let $f=\mathcal{I}(F)$, $g=\mathcal{I}(G)$ both belonging to $\mathcal{S}(\Omega_{D})$ then the \textit{slice product} of $f$ and $g$ is the slice function
\begin{equation*}
f* g:=\mathcal{I}(FG)\in\mathcal{S}(\Omega_{D}).
\end{equation*}
\end{definition}
Explicitly, if $F=F_{1}+F_2\imath $ and $G=G_{1}+ G_{2}\imath$ are stem functions, then
\begin{equation*} 
FG=F_{1}G_{1}-F_{2}G_{2}+(F_{1}G_{2}+F_{2}G_{1})\imath .
\end{equation*}

\begin{definition}
A slice function $f=\mathcal{I}(F)
\in {\mathcal S}(\Omega_D)$ 
is called \textit{slice-preserving} if, 
for all $J\in \S $, $f(\Omega_{D}\cap\mathbb{C}_{J})\subset \mathbb{C}_{J}$.
\end{definition}

Slice preserving functions satisfy the following characterization.
\begin{proposition}\label{slpreschar}
Let $f=\mathcal{I}(F_{1}+F_2\imath )$ be a slice function. Then $f$ is slice preserving if and only if the $\H $-valued components $F_{1}$, $F_{2}$ are real valued.
\end{proposition}
Since real numbers commute with all quaternions, this has the
following consequence:

Let $f, g\in\mathcal{S}(\Omega_{D})$. If $f$ is slice preserving, then
\begin{equation*}
(f*g)(x)=f(x)g(x).
\end{equation*}
If $f$ and $g$ are both slice-preserving, then
 $fg=f*g=g*f=gf$.

As stated in \cite{genstostru}, if
$f$ is a  regular function defined on $\mathbb{B}_{\rho}$, the ball of center 0 and radius $\rho,$ then it is slice preserving if and only if $f$ can be expressed as a power series of the form
\begin{equation*}
f(x)=\sum_{n\in\mathbb{N}}x^{n}a_{n},
\end{equation*}
with $a_{n}$ real numbers.

The following definitions are taken from \cite{genstostru,ghiloniperotti}.
\begin{definition}
Let $f=\mathcal{I}(F)\in\mathcal{S}(\Omega_{D})$, then also $F^{c}(z)=F(z)^{c}:=F_{1}(z)^{c}+F_2(z)^{c}\imath $ is a stem function.
We set
\begin{itemize}
\item $f^{c}:=\mathcal{I}(F^{c})\in\mathcal{S}(\Omega_{D})$, the \textit{slice conjugate} of $f$;
\item $f^{s}:=f^c* f$, the \textit{symmetrization} of $f$.
\end{itemize}
\end{definition}

\begin{remark}
We have that $(FG)^{c}=G^{c}F^{c}$, and so $(f* g)^{c}=g^{c}* f^{c}$.
In particular,
$f^{s}=(f^{s})^{c}$.
Moreover it holds
  \begin{equation*}
   (f* g)^{s}=(f)^{s}(g)^{s}\quad \mbox{and }\quad (f^{c})^{s}= f^{s}.
  \end{equation*}
  Another observation is that, if $f$ is slice preserving,
  then $f^{c}=f$ and so $f^{s}=f^{2}$.
\end{remark}
Frequently, the sum $f+f^c$ is denoted by $Tr(f)$.

\subsubsection{Zeros of regular functions}
\label{sect-zero}
We are going now to recall some key facts about the zeros of a slice function.

Let $f:\Omega_{D}\rightarrow \H $ be any slice function with
zero locus
\[
\mathcal{Z}(f) = \{ x\in\Omega_D : f(x)=0 \}.
\]
Let $x\in\mathcal{Z}(f)$.
There are the following three possibilities:

\begin{itemize}
\item $x\in\mathbb{R}$, i.e., $x$ is a \textit{real} zero;
\item $x$  a \textit{punctual (non-real)} zero, i.e.,  
$x\notin\mathbb{R}$ and $\S _{x}\cap\mathcal{Z}(f)=\{x\}$;
\item $x$ a \textit{spherical zero}, i.e., 
$x\notin\mathbb{R}$ and $\S _{x}\subset\mathcal{Z}(f)$.
\end{itemize}

The inclusion
\begin{equation}\label{zerosofprod}
\mathcal{Z}(f)\subset\mathcal{Z}(f*g),
\end{equation}
holds for any two slice functions $f,g:\Omega_{D}\rightarrow\H $,
while in general $\mathcal{Z}(g)\not\subset\mathcal{Z}(f*g)$.

What is true in general is the following equality:
\begin{equation*}
\bigcup_{x\in \mathcal{Z}(f* g)} \S _{x}=\bigcup_{x\in  \mathcal{Z}(f)\cup  \mathcal{Z}(g)}\S _{x}
\end{equation*}

\begin{theorem}[\cite{genstostru}]
Let $f\in\mathcal{SR}(\Omega_{D})$. If $\S _{x}\subset\Omega_{D}$ then the zeros of $f^{c}$ on
$\S _{x}$ are in bijective correspondence with those of $f$. Moreover $f^{s}$ vanishes exactly on the sets $\S _{x}$
on which $f$ has a zero.
\end{theorem}

\subsubsection{Identity principle}

\begin{theorem}[Identity principle, \cite{gentilistruppa1}, \cite{gentilistruppa}]\label{identity}  
 Let $\Omega_D$ be a slice domain. Given $f=\mathcal{I}(F):\Omega_{D}\rightarrow \H $ a regular function, if there exists a $J\in\S $ such that  $(\Omega_{D}\cap\mathbb{C}_{J})\cap\mathcal{Z}(f)$ admits an accumulation point, then $f\equiv 0$ on $\Omega_{D}$.
 \end{theorem}

\begin{corollary}\label{cor-id}
Let $f$ be a regular function on a circular slice domain $\Omega_D$.

If there exists a convergent sequence
of distinct numbers $p_n=x_n+iy_n$ in $D$
such that $f$ has at least one zero on every sphere of the form
\[
\Sigma_n=\{ x_n+Iy_n: I\in\S\},
\]
then $f$ vanishes identically.
\end{corollary}

\begin{proof}
Under the assumptions of the corollary, the symmetrization
$f^s=f^c*f$ vanishes identically on each $\Sigma_n$. Hence
$(\Omega_{D}\cap\mathbb{C}_{J})\cap\mathcal{Z}(f^s)$ contains
an accumulation point for any $J\in\S$. Consequently $f$
must vanish identically.
\end{proof}

\subsubsection{Multiplicities of zeros}
Let $f\in\mathcal{SR}(\Omega_D)$ such that $f^{s}$ does not vanish identically. Given $n\in\mathbb{N}$ and $q \in \mathcal{Z}(f)$,
we say that $x$ is a zero of $f$
of \textit{total multiplicity} $n$, and we will denote it by $m_f(x )=n$,
if $((q-x )^{s})^n\mid f^{s}$ and $((x-q )^{s})^{n+1}\nmid f^{s}$. 
 If $m_f(x )=1$, then $x $ is called a \textit{simple zero} of $f$.

\begin{lemma}\label{right-factor}
Let $f$ be a regular function on a circular domain
with $f(p)=0$. Then there exists
$\tilde p\in\S_p$ and a regular function $g$
such that
$f(q)=g(q)*(q-\tilde p)$.
\end{lemma}

\begin{proof}
There is an element $a\in\S_p$ such that $f^c(a)=0$,
implying that there exists a regular function $h$
with
$f^c(q)=(q-a)*h(q)$. It follows that
$f(q)=h^c(q)*(q-a^c)$.
\end{proof}

\subsubsection{Semi$-$regular functions and their poles}
We will recall now some concept of ``semi-regular functions''
which are the quaternionic analog of  meromorphic functions.
Here our main references are
\cite{genstostru} and \cite{ghilperstosing}. 

\begin{definition}
  Let $f=\mathcal{I}(F)\in\mathcal{SR}(\Omega_{D})$.
  We call the \textit{slice reciprocal} of $f$ the slice function
\begin{equation*}
f^{-*}:\Omega_{D}\setminus \mathcal{Z}(f^{s})\rightarrow \hh, \qquad f^{-*}=\mathcal{I}((F^cF)^{-1}F^c).
\end{equation*}
\end{definition}
From the previous definition it follows that, if $x\in\Omega_D\setminus \mathcal{Z}(f^{s})$, then
\begin{equation*} 
f^{-*}(x)=(f^{s}(x))^{-1}f^{c}(x).
\end{equation*}
The regularity of 
$f^{-*}$ on $\Omega_{D}\setminus \mathcal{Z}(f^{s})$
 just defined follows thanks to the last equality.

If $f$ is slice preserving, then $f^{c}=f$ and so $f^{-*}=f^{-1}$ where it is defined. Moreover $(f^{c})^{-*}=(f^{-*})^{c}$.

\begin{proposition}
Let $f\in\mathcal{SR}(\Omega_{D})$ such that $\mathcal{Z}(f)=\emptyset$, then $f^{-*}\in\mathcal{SR}(\Omega_{D})$ and 
\begin{equation*}
f* f^{-*}=f^{-*}* f=1.
\end{equation*}
\end{proposition}

The concept of a semi-regular function has been
  introduced in \cite{ghilperstosing}, 11.1-11.2.
  For our purposes the crucial property of semi-regular functions
  is that every semi-regular function $f$ may locally written in the
  form $F=g^{-*}*h$ with $g,h$ being slice regular functions.

\begin{lemma}\label{lem-4.1}
Let $f$ be a slice function, $x,y\in\mathbb{R}$, $I,J\in\S$.

Then
\[
f(x+yI)+f(x-yI)=
f(x+yJ)+f(x-yJ)
\]
\end{lemma}
\begin{proof}
Due to the representation formula we have
\[
f(x+yJ)=\frac{1-JI}2f(x+yI)+ \frac{1+JI}2f(x-yI)
\]
and 
\[
f(x-yJ)=\frac{1+JI}2f(x+yI)+ \frac{1-JI}2f(x-yI)
\]
Adding both above equalities yields the assertion of the lemma.
\end{proof}

\section{Divisors}\label{sect-div}
In complex analysis, the divisor of a holomorphic function is
the formal sum of its zeroes, counted with the respective multiplicities.
We propose that for a quaternionic slice regular function defined on $\Omega_D$
the divisor should
be defined as a formal sum of points in the closed upper plane intersected
with $D$, i.e., on $\{z\in D:\Im(z)\ge 0\}$.

\begin{definition}
  Let $\Omega_D$ be a slice domain
  and let $f$ be a slice regular
function on $\Omega_D$. Let $D^+=D\cap\{ z\in\C: \Im(z)\ge 0\}$.

Then the ``(slice) divisor'' $div(f)$ of $f$ is defined as the
formal $\Z$-linear combination $\sum_{z\in D^+} m_z(f)\{z\}$
where for $z=x+yi$ the multiplicity $m_z(f)$ is defined as follows:
$m_z(f)=m$ if in a neighbourhood of $\S_{x+yI}=
\{x+yJ:J\in\S\}$ the function $f$ can be written as
\[
f(q)=(q-a_1)*\ldots*(q-a_m)*g(q)
\]
with $a_i\in \S_{x+yI}$ and $g$ being a slice regular function without zeros
on $\S_{x+yI}$.
\end{definition}

Standard facts on zeros of slice regular functions (see \ref{sect-zero})
guarantee us the following properties:

\begin{itemize}
\item
$div(f*g)=div(f)+div(g)$ if both $f$ and $g$ are slice regular on $\Omega_D$.
\item
If $p_k=a_k+I_kb_k$
(with $a_k,b_k\in\R, b_k\ge 0, I_k\in\S $) are the isolated zeros with multiplicity 
$n_k$ and $\S_{c_k+Jd_k}$ are the spherical zeros with multiplicity
$m_k$, then
\[
div(f) = \sum_k n_k\{a_k+ib_k\} +2 m_k\{c_k+id_k\}
\]
\item
  $\{z\in D^+: div(f)>0\}$ is discrete in $D^+$
  (for $f\not\equiv 0$).
\end{itemize}

For example, let $I,J\in\S$ with $I\ne J$ and
consider $f(q)=(q-I)*(q-J)=q^2-q(I+J)+IJ$.
Then $f$ has a zero
only at $I$ while $g(q)=(q-J)*(q-I)$ has a zero only at $J$, but the
divisor is the same:
\[
div(f)=div(q-I)+div(q-J)=div(g)=2\{i\}.
\]

This notion of a divisor is easily extended
from (slice) regular to semi-regular functions,
since semi-regular functions may locally be written in the form
$f=g^{-*}h$ with $g,h$ slice regular.
If $z=x+yi$ is a point in a  symmetric domain $D$, and $f$ is semi-regular
on $\Omega_D$ we choose a sufficiently small symmetric
domain $D'$ with $p\in D'\subset D$ such that $f$ 
may be written in the
form $g^{-*}*h$ on $\Omega_{D'}$.
Then we define $div(f)=div(h)-div(g)$ on $D'$.

\begin{warning}
  In complex analysis, a meromorphic function $f$ is holomorphic
  iff $div(f)\ge 0$. The analog quaternionic statement is {\em not}
  true. For example, let $I,J\in\S$ and
  consider
  \[
  (q-I)*(q-J)\frac{1}{q^2+1}= (q^2-q(I+J)+IJ)\frac{1}{q^2+1}
  \]
  This is a semi-regular function whose divisor is zero,
  although $f$ is not slice regular unless $I=-J$.
\end{warning}

\section{A mean value theorem}
\begin{proposition}[General Mean Value Formula]\label{GMVF}
  \label{gen-mean}
  Let $\mu$ be a probability measure on $\S$
  which is
  invariant under $J\mapsto -J$.
  
  Let $f$ be a slice regular
  function on a slice domain $\Omega_D$
  induced by a stem function $F:D\to\HC$.
  Let $a,b,r\in\R$, $b\ge 0$, $r>0$ and $I\in\S$
  such that $\Omega_D$ contains the closed ball with radius $r$
  and center $a+bI$.
 
  Then we have:
  \begin{align*}
  F_1(a+bi)&=\frac{1}{4\pi}
  \int_\S\int_0^{2\pi}
  f(a+bJ+re^{J\theta})+f(a-bJ+re^{-J\theta})
  d\theta d\mu(J)\\
  &=\frac{1}{2\pi}
  \int_\S\int_0^{2\pi}
  f(a+bJ+re^{J\theta})
  d\theta d\mu(J)
  \end{align*} 
  as well as
  \[
  F_2(a+bi)= -\frac{1}{2\pi}
  \int_\S J \int_0^{2\pi}
  f(a+bJ+re^{J\theta}))
  d\theta d\mu(J)
  \]
  and therefore
  \begin{align*}
    f(a+bI)& = F_1(a+bi)+IF_2(a+bi)\\
    &=\frac{1}{2\pi}
  \int_\S  \int_0^{2\pi}
  (1-IJ)f(a+bJ+re^{J\theta})
  d\theta
  d\mu(J)\\
  \end{align*}
\end{proposition}
\begin{proof}
  This follows from combining the complex mean value theorem
  with the formulae relating slice and stem functions.
\end{proof}

\begin{corollary}\label{meanvalue}
Let $f$ be a regular function, $r>0$, $a\in\R$

Then
\[
f(a)=\frac1{2\pi}\int_{\S}\int_0^{2\pi}
f(a+r\cos\theta+r\sin\theta I)
d\theta d\mu(I)
\]
for any probability measure  $\mu$  on $\S$
which is invariant under the involution $J\mapsto -J$.
\end{corollary}

\begin{proof}
  This a special case of Proposition~\ref{gen-mean} with $b=0$.
\end{proof}

\begin{remark}
  Note that in Corollary~\ref{meanvalue} we integrate over the
  sphere with radius $r$ and center $a$, but not with respect to the
  euclidean volume element $dV$ on the $3$-sphere.

  This is crucial.
  
For example, $\int_{||q||=1} q^2dV<0$, hence
$\int_{||q||=1} f(q)dV\ne f(0)$ for $f(q)=q^2$.
\end{remark}

\subsection{Characterization of harmonicity}
A continuous function on $\C$ is harmonic if and only if
it satisfies the mean value property.

We derive a similar criterion in the quaternionic setup.

\begin{theorem}\label{char-harm}
  Let $\mu$ be a probability measure on $\S$ which is
  invariant under the involution $J\mapsto -J$.
  Let $f:\H\to\H$ be a continuous function and for $p=a+bI\in\H$
  and $r>0$ define
  \[
  M_{p,r}=
  \frac{1}{2\pi}
  \int_\S  \int_0^{2\pi}
  (1-IJ)f(a+bJ+re^{J\theta})
  d\theta
  d\mu(J).
  \]

  Then $f$ is harmonic (in the sense of being the
  sum of a regular and an anti-regular function) if and only if
\begin{equation}\label{eq-mr}
  M_{p,r}=f(p)\quad\quad \forall p\in\H, r>0.
\end{equation}
\end{theorem}

\begin{proof}
We assume that \eqref{eq-mr} holds.
  
  The function $f$ is a slice function if and only if it satisfies the
  representation formula.
  Hence $f$ is a slice function iff
  \[
  f(a+bH)=\frac{1-HI}2 M_{p,r}+\frac{1+HI}2 M_{p^c,r}
  \]
  for all $a,b\in\R$, $H,I\in\S$ and  $p=a+bI$.
  
  This can be verified by explicit calculation:
  \begin{align*}
    & \frac{1-HI}2 M_{p,r}+\frac{1+HI}2 M_{p^c,r}\\
    & = 
  \frac{1}{4\pi}
  \int_\S  \int_0^{2\pi}
  \left((1-HI)(1-IJ)+(1+HI)(1+IJ)\right)
  f(a+bJ+re^{J\theta})d\theta d\mu(J)\\
    & = 
  \frac{1}{2\pi}
  \int_\S  \int_0^{2\pi}
  \left( 1-HJ \right)
  f(a+bJ+re^{J\theta})d\theta d\mu(J)\\
  &= M_{a+bH,r}=f(a+bH).\\
  \end{align*}

  Thus $f$ is a slice function induced by some stem function $F$.
  This stem function can be easily determined as $F=F_1+F_2\tensor i$ with
  \begin{align*}
    F_1(a+bi) &
=      \frac{1}{2\pi}
  \int_\S  \int_0^{2\pi}
  f(a+bJ+re^{J\theta})d\theta d\mu(J)\\
    F_2(a+bi) &
=     - \frac{1}{2\pi}
  \int_\S  \int_0^{2\pi}
  Jf(a+bJ+re^{J\theta})d\theta d\mu(J)\\
  \end{align*}
  Since $\mu$ is invariant under $J\mapsto -J$, we have
  \begin{align*}
    F_1(a+bi) &
=      \frac{1}{4\pi}
  \int_\S  \int_0^{2\pi}
  f(a+bJ+re^{J\theta}) + f(a-bJ+re^{-J\theta})
  d\theta d\mu(J)\\
&    = \frac{1}{2\pi}\int_0^{2\pi} F_1(a+bi+re^{i\theta}) d \theta.\\ 
  \end{align*}
  Thus $F_1$ satisfies the ordinary mean value property for functions
  defined on $\C$ and therefore must be harmonic. Similar arguments
  apply to $F_2$.
  As a result we see that $F$ is the sum of a holomorphic
  and an antiholomorphic function from $\C$ to $\H\tensor_{\R}\C$ and
  consequently $f:\H\to\H$ is the sum of a regular and an anti-regular
  function.

  For the opposite direction, assume that $f$ is the sum of a regular
  function and an anti-regular function.
  Then \eqref{eq-mr} follows immediately from
  Proposition~\ref{gen-mean}.
  \end{proof}

\section{Generalized Representation Formula}
For slice regular functions, the formula below
already appeared in \cite{CGSS}: see Theorem 3.2.
Here we give a new proof and we deduce some consequences.

\begin{proposition}\label{gen-rep-formel}
  Let $f$ be a slice function (not necessarily regular)
  and let $I,J,H\in\S$
(not necessarily orthogonal). Assume that $J\ne I$, $H\ne I$.
Then the following equality holds:
\begin{gather*}
 \left( \frac{1+JI}2 \right)^{-1}
f(x+yJ) - \left(\frac{1+HI}2\right)^{-1}f(x+yH) \\
=
\left( \left( \frac{1+JI}2 \right)^{-1}
\left( \frac{1-JI}2 \right) 
- \left(\frac{1+HI}2\right)^{-1}
\left(\frac{1-HI}2\right)\right) f(x+yI)
\end{gather*}
\end{proposition}

\begin{proof}
We have
\[
f(x+yJ)=\frac{1-JI}2f(x+yI)+ \frac{1+JI}2f(x-yI)
\]
and
\[
f(x+yH)=\frac{1-HI}2f(x+yI)+ \frac{1+HI}2f(x-yI).
\]
A linear combination of both equations yields:
\begin{gather*}
\left(\frac{1+HI}2\right) \left( \frac{1+JI}2 \right)^{-1}
f(x+yJ) - f(x+yH) \\
=
\left(\left(\frac{1+HI}2\right) \left( \frac{1+JI}2 \right)^{-1}
\left( \frac{1-JI}2 \right) 
- \left(\frac{1-HI}2\right)\right) f(x+yI)
\end{gather*}
and
\begin{gather*}
 \left( \frac{1+JI}2 \right)^{-1}
f(x+yJ) - \left(\frac{1+HI}2\right)^{-1}f(x+yH) \\
=
\left( \left( \frac{1+JI}2 \right)^{-1}
\left( \frac{1-JI}2 \right) 
- \left(\frac{1+HI}2\right)^{-1}
\left( \frac{1-HI}2\right)\right) f(x+yI)
\end{gather*}
\end{proof}

\begin{lemma}\label{lem-inj}
Fix $I\in\S$. For $J\ne I$ we define 
\[
R(J)=\left( \frac{1+JI}2 \right)^{-1}
\left( \frac{1-JI}2 \right)
\]

Then $R:\S\setminus\{I\}\to\H $ is injective, and
$R(J)=0$ iff $J=-I$.

Furthermore $\lim_{J\to I}|R(J)| = +\infty$.
\end{lemma}

\begin{proof}
Since $I,J$ are purely imaginary, we have $\overline{JI}=IJ$.
Therefore
\[
\left( \frac{1+JI}2 \right)^{-1}=2\frac{1+IJ}{|1+JI|^2}
\]
and therefore
\[
R(J)= \frac{(1+IJ)(1-JI)}{|1+JI|^2}
=\frac{IJ-JI}{|1+JI|^2}
\]
Since $\overline{IJ}=JI$, $\frac{IJ-JI}{2}=\frac{IJ-\overline{IJ}}{2}$
denotes the vector part of $IJ$.
Define $r=\Re  \, (IJ)$. Observe that $r\in]-1,+1]$.
Using $|IJ|=1$ we know that the vector part of $IJ$ has norm
$\sqrt{1-r^2}$.
Therefore
\[
|R(J)|^2= 4\frac{1-r^2}{|1+JI|^4}
\]
Now $|1+JI|^2=|1+\Re  \, (JI)|^2+|\Im  \, (JI)|^2$ and
therefore $|1+JI|^4=(2+2r)^2$ implying
\[
|R(J)|^2=4\frac{1-r^2}{(2+2r)^2}=\frac{(1+r)(1-r)}{(1+r)^2}
=\frac{1-r}{(1+r)}=\left( -1 +\frac{2}{1+r}\right).
\]
We observe that 
the map
\[
r\mapsto \left( -1 +\frac{2}{1+r}\right)
\]
is evidently an injective map from $(-1, +1]$ to $\mathbb{R}^+$.

As a consequence, we obtain: If $|R(J)|=|R(H)|$ for some $J,H\in\S \setminus \{ I \} $,
then $\Re  \, (IJ)=\Re  \, (HI)$. 
On the other hand, the vector part of $IJ$ equals
\[
\frac{1}{2}\left(IJ-\overline{IJ}\right)
=\frac12 R(J) |1+JI|^2 =  R(J) (1+r)
\]
because $|1+JI|^2=(2+2r).$
Hence also the vector parts of $JI$ and $HI$ have to agree
as soon as $R(J)=R(H)$. Finally observe that $J=H$ if $JI=HI$.
\end{proof}

\begin{proposition}\label{gen-repr}
Fix $I\in\S$.
Then there exists a continuous map $M=(M_1,M_2):\S\times\S\setminus
D_\S\to\H \times\H $ (where $D_\S$ denotes the diagonal, i.e.,
$D_\S=\{(q,q):q\in\S\}$)
such that
\begin{equation} \label{repgenfor}
f(x+yI)=M_1(J,H)f(x+yJ) + M_2(J,H) f(x+yH) \ \forall J,H\in\S
\end{equation}
for every regular function $f$.
\end{proposition}

\begin{proof}

First assume that $I,J,H$ are pairwise distinct.

Then the statement follows from Proposition \ref{gen-rep-formel}
with
\[
M_1(J,H)=\left( R(J)-R(H) \right)^{-1} \left( \frac{1+JI}2 \right)^{-1}
\]
and
\[
M_2(J,H)=- \left( R(J)-R(H) \right)^{-1} \left( \frac{1+HI}2 \right)^{-1}.
\]
(Note that $1+JI\ne 0$, resp.~$1+HI\ne 0$, because of our assumptions
$J\ne I$, $H\ne I$. Note further that $R(J)-R(H)\ne 0$
due to $J\ne H$ and the injectivity statement of Lemma \ref{lem-inj}.)

Next we claim that the functions $M_i$ do extend continuously
to the points where $J=I$ or $H=I$, i.e., extend continuously to all
of $\S\times\S\setminus
D_\S$.

Consider the case where $J$ approaches $I$. Since
we excluded the diagonal $D_\S$, we may fix $H\ne I$.

Now
\begin{align*}
M_1(J,H) &=\left( R(J)-R(H) \right)^{-1} \left( \frac{1+JI}2 \right)^{-1}\\
&= \left( \frac{1+JI}2  \left(  R(J)-R(H) \right)  \right)^{-1}\\
&= \left( \frac{1+JI}2  R(J) - \frac{1+JI}2R(H)   \right)^{-1}\\
&= \left( \frac{1-JI}2 - \frac{1+JI}2R(H) \right)^{-1} \\
\end{align*}
implying
\[
\lim_{J\to I} M_1(J,H)
= \left( \frac{2}2 - 0\cdot R(H) \right)^{-1} =1
\]
In a similar way one proves
\[
\lim_{J\to I} M_2(J,H)=0
\]
and
the analog statement for $H\to I$.
\end{proof}

\begin{remark}
We observe that our formula \eqref{repgenfor} coincides with the Representation Formula of Proposition 6 in \cite{ghiloniperotti},
when $M_1(J,H)=(I-H)(J-H)^{-1}$ and $M_2(J,H)=-(I-J)(J-H)^{-1}.$
\end{remark}

\section{Rotations}
\label{sect-rot}
For every $w\in\H ^*$ let $S_w:\H \to\H $ denote the map given by
$S_w(q)=w^{-1}qw$. This is an orthogonal transformation of $\mathbb{R}^4$ which fixes
$\mathbb{R}$ pointwise.
Observe that $S_w^{-1}=S_{w^{-1}}$.

\begin{lemma}
Let $I,J,K$ be orthogonal imaginary units.

Then $S_I:\H \to\H $ is a linear map acting as $id$
on $\C_I=\left<1,I\right>_\mathbb{R}$
and as $-id$ on $\C_I^\perp= \left<J,K\right>_\mathbb{R}$.
\end{lemma}

\begin{proof}
  Follows easily from explicit calculations.
\end{proof}

  The following lemma is a well known result, see for example
  \cite{GHS} Prop. 2.22, page 28), but for the reader convenience
  we prefer to give here our own proof.

\begin{lemma}\label{generate-so}
The group of all orientation preserving orthogonal transformations
of $\left<I,J,K\right>_\mathbb{R}$ is generated by the
transformations $S_w$ with $w\in\S$.
\end{lemma}

\begin{proof}
The group is $SO(3,\mathbb{R})$. 
For each $k\in\N,$ let $\Sigma_k$ denote the
set of all $S_{w_1}\circ\ldots\circ S_{w_{2k}}$.
Then $\Sigma=\cup_k\Sigma_k$ is the group generated by all the $S_w$.
($\Sigma$ is evidently a semigroup and in fact a group, because
$(S_w)^{-1}=S_{w^{-1}}$.)
$\Sigma$ is connected, because each $\Sigma_k$ is connected and $\Sigma_{k} \subseteq \Sigma_{k+1}.$
On the other hand,
it is not commutative, since e.g.~$S_I$ and $S_{(I+J)/\sqrt 2}$
do not commute.
However, by standard Lie theory, $SO(3,\mathbb{R})$ has no non-commutative
connected subgroups except $SO(3,\mathbb{R})$ itself.
\end{proof}

\begin{lemma}\label{aver-sw}
Let $q\in\H$.
Let $\mu$ denote the (unique) probability measure on $\S$ which is
invariant under all rotations.

Then
\[
\Re  \, (q)=\int_{\S} S_w(q)d\mu(w)
\]
\end{lemma}

\begin{proof}
The map $H:q\mapsto \int_{\S} S_w(q)d\mu(w)$ is $\R$-linear.

Let $v\in\S$. Then $S_v$ is an orthogonal transformation.
Due to the invariance of the measure $\mu$, we have:
\[
S_v(H(q))
=\int_{\S} S_v\left(S_w(q)\right)d\mu(w)
=\int_{\S} S_w(q) S_v^*d\mu(w)=\int_{\S} S_w(q)d\mu(w)=H(q).
\]
(Here $S_v^*$ denotes the pull-back by the map $S_v$.)
 
It follows that 
\[
H(q)\in\{x\in\H:S_w(x)=x ,\,\, \forall w \in \mathbb{S} \}=\R\, ,  \,\, \forall q\in\H.
\]
We observe that $\int_{\S} S_w(q)d\mu(w)=q, \,\, \forall q\in\R$,
because $S_w(q)=q, \,\,  \forall q\in\R, w\in\S$.

Now let $q$ be in the orthogonal complement of $\R$, i.e., in the
real vector subspace $V$ of $\H$ spanned by $I,J,K$.
Since the integral is linear, and $V$ is stabilized by every $S_w$
($w\in\S$)
it follows that
\[
H(q)=\int_{\S} S_w(q)d\mu(w)\in V , \,\,\, \forall  q\in V.
\]
Combined with the fact $H(q)\in\R,  \,\, \forall q\in\H$,
we obtain
\[
H(q)=\int_{\S} S_w(q)d\mu(w)\in V\cap\R=\{0\}, \,\,\,  \forall q\in V.
\]
Thus 
\[
\Re  \, (q)=\int_{\S} S_w(q)d\mu(w)
\]
for every $q\in \R$ and every $q\in V$. By $\R$-linearity of the map $H$,
it follows that this equality holds for all $q\in\H$.
\end{proof}

\begin{definition}\label{def-rot}
A function $f:\H\to\H$ is called
{\em ``rotationally invariant''} if it is invariant under all 
orthogonal transformations of the space of imaginary elements.
\end{definition}

\begin{remark} This class of functions has been studied in
  \cite{GMP} where they are called {\em ``circular''} functions.
\end{remark}

\begin{lemma}\label{LemRotInv}
For a function $f:\H\to\H$ the following properties are
equivalent:
\begin{enumerate}
\item
$f$ is {\em rotationally invariant}.
\item
$f(q)=f(S_wq)$ for all $q\in\H$, $w\in\S$.
\item
$f(x+yI)=f(x+yJ)$ for all $x,y\in\R$, $I,J\in\S$.
\item
$f$ is induced by a stem function $F$ with $F(\C)\subset\H\tensor\R$,
i.e., $F_2=0$ for $F=F_1+F_2\imath $.
\end{enumerate}
\end{lemma}

\begin{proof}
First we show $(3)\Rightarrow(4)$: Given such a function $f$, we define
$F:\C\to \H\tensor\C$ as $F(x+iy)=f(x+yI)\tensor 1$ for any $I\in\S$.
Since $(3)$ implies $f(x+yI)=f(x+y(-I))$, we have $F(z)=F(\bar z)$.
Combined with $F(\C)\subset\H\tensor \R$ it follows
that $\overline{F(z)}=F(z)=F(\bar z)$.
Consequently $F$ fulfills $F(\bar z)=\overline{F(z)}$ and is a
stem function.

The implications $(4)\Rightarrow (3)\iff(1)\Rightarrow (2)$ are obvious.
The implication $(2)\Rightarrow(1)$
 follows from Lemma~\ref{generate-so}.
\end{proof}

\begin{definition}\label{defRw}
    Let $\Omega$ be a circular\footnote{in the sense of
      Definition~\ref{def-circular}} subset of $\H$ and let
    $f:\Omega\to \H$ be a function. Let $w\in\S$.

    Then we define a function $R_wf:\Omega\to \H$ as
    \[
    (R_wf)(q)=S_w^{-1}\left(f(S_w(q)\right)).
    \]
\end{definition}

\begin{lemma}\label{lemstemrot}
  Let $\Omega_D$ be a slice domain, and let
  $f:\Omega_D \to\H $ be a slice function induced by a stem function 
$F:D\to\H \tensor\mathbb{C}$.

  Then $R_wf$ is induced by $S_{w^{-1}}(F)$ with $S_{w^{-1}}$ 
acting via the first
factor of the tensor product $\H \tensor\mathbb{C}$.
\end{lemma}

\begin{proof}
We have
\[
f(x+yI)=F_1(x+yi)+IF_2(x+yi)\ \ \ (x,y\in\mathbb{R}, I\in\S,
x+yi\in D)
\]
and 
\[
R_wf(q)=w\left( f(w^{-1}qw) \right)w^{-1}.
\]
Now
$w^{-1}qw=x+yw^{-1}Iw$ for $q=x+yI$.
Furthermore  $w^{-1}Iw\in\S$ and consequently
\[
f(x+yw^{-1}Iw)=
F_1(x+yi) +w^{-1}IwF_2(x+yi).
\]
Therefore
\begin{align*}
(R_wf)(x+yI)&=
w\left(f\left(x+yw^{-1}Iw\right)\right)w^{-1}\\
&=
w\left(
F_1(x+yi) +w^{-1}IwF_2(x+yi)
\right)w^{-1}\\
&= wF_1(x+iy)w^{-1}+IwF_2(x+yi)w^{-1}
\end{align*}
Therefore (using Lemma~\ref{stem-basic})
$R_wf$ is induced by the stem function
\[
S_{w^{-1}}F=\left(S_{w^{-1}}F_1\right)
\tensor 1 + 
\left(S_{w^{-1}}F_2\right)\tensor i 
\]
\end{proof}

\begin{corollary}
  Let $f$ denote a slice regular function
  on a slice domain $\Omega_D$ and let $w\in\S$.

  Then $R_wf:\Omega_D\to\H$ is a slice regular function, too.
\end{corollary}

\begin{proof}
  As a slice regular function, $f$ is induced by a holomorphic
  stem function $F:D\to\HC$. Holomorphicity of $F$ implies that
  \[
  z\mapsto \left( S_wF\right)(z)=w^{-1}(F(z))w
  \]
  is likewise holomorphic. Hence $R_wf$ is slice regular.
  \end{proof}

\begin{corollary}
Under the assumptions of the Lemma \ref{lemstemrot}, 
\[
g=\int_{\S} R_w(f)d\mu(w)
\] is
induced by the stem function $R(F)$ where $R$ denotes the
real part in the first factor of the tensor product, i.e.,
$R(a\tensor b)=(\Re  \, a)\tensor b$,
for $a\in\H$, $b\in\C$.
\end{corollary}

\begin{proof}
By the Lemma \ref{lemstemrot}, $g$ is induced by $\int_{w} S_{w^{-1}}Fd\mu(w)$.
Aided by Lemma~\ref{stem-basic}, this implies the assertion, because 
\[
\int_{\S} S_{w^{-1}}(q)d\mu(w)=\Re  \, (q)
\] for every $q\in\H $ (Lemma~\ref{aver-sw}).
\end{proof}

\begin{lemma}
  Let $f$ be a slice regular function
  given by a convergent power series
  $f=\sum_{k=0}^{+ \infty} q^ka_k$.

  Then the following holds:  
\[
\int_\S R_wfd\mu(w) =\sum_k q^k\Re  \, (a_k)=\frac{1}{2}(f+f^c) = \frac{1}{2} Tr(f).
\]
\end{lemma}

We observe that $R_{vw}=R_v\circ R_w$ and $S_{vw}=S_w\circ S_v$
for $v,w\in\H ^*$.

\begin{definition}\label{dirder}
  For $v\in\H$ let $\partial_v$ denote the directional derivative in the
direction of $v$, i.e.,
\[
(\partial_v)(f)(q)=\lim_{t\to 0, t\in\R^*}\frac{f(q+vt)-f(q)}{t}.
\] 
\end{definition}

Next we discuss
differential operators $\partial_*$, $\bar\partial_*$ for
differentiable functions on $\H\setminus\R$.
These operators were first introduced
  by Ghiloni and Perotti
  in \cite{{ghilperglobal}}
  as $\vartheta$ resp.~$\bar\vartheta$. For slice functions, they coincide
  with the operators $\partial/\partial x$
and $\partial/\partial x^c$ introduced 
 in \cite{ghiloniperotti}.

\begin{definition}
Let $\Omega$ be a  domain in $\H$
and let  $f:\Omega\to\H$ be a $C^1$-function.
Let $I\in\S$ and $q\in(\C_I\setminus\R)\cap\Omega$.

We define
\begin{align*}
\left(\partial_*f\right)(q)     &= 
\left(\frac 12\left(\partial_1-I\partial_I\right) f\right)(q)\\
\left(\bar\partial_*f\right)(q) &=
\left(\frac 12\left(\partial_1+I\partial_I\right) f\right)(q).
\end{align*}
(with $\partial_1$, $\partial_I$ being directional derivatives,
cf.~\ref{dirder}).
\end{definition}

\begin{remark}
  Let $\Omega$ be a domain.
  Let $I,J$ be orthogonal imaginary units, $f:\Omega\to\H$,
  $g,h:\Omega\to\C_I$ be $C^1$-functions with $f=g+hJ$.

  Then $\bar\partial_* f$ vanishes on $\Omega\cap\C_I$ if
  and only if the restrictions of $g,h$ are holomorphic
  functions from $\Omega\cap\C_I$ to $\C_I$.
\end{remark}

Next we define a Laplacian:

\begin{definition}\label{def-laplace-star}
  Let $\Omega$ a domain in $\H$, $f$ be a $C^2$-function on $\Omega$
  and $q\in\Omega\setminus\R$.

  We define
  \[
  (\Delta_* f)(q)=\left( 4\partial_*\bar\partial_* f\right)(q).
  \]
\end{definition}

\begin{remark}
  \begin{enumerate}
  \item
    $  \Delta_* = 4\partial_*\bar\partial_* = 4\bar\partial_*\partial_*$.
  \item
    For orthogonal imaginary units $I,J\in\S$ the restriction of $\Delta_*f$ to
    $\Omega\cap\C_I$ vanishes if and only if
    $f|_{\C_I}=g+hJ$ with $g,h:\C_I\to\C_I$ harmonic in the sense
    of complex analysis.
  \end{enumerate}
\end{remark}

\begin{lemma}
Let $I\in\S$. Then $\Delta_*(f)$ vanishes along $\mathbb{C}_I$
for (slice) regular functions $f$ (and anti-regular functions $f$).
\end{lemma}

\begin{proof}
This is clear, because the restriction of a regular function $f$ to a complex
line $\C_I$ can be written as $f(z)=f_1(z)+f_2(z)J$ 
where $f_i:\C_I\to\C_I,$ for $i=1,2,$ are
entire functions with respect to the complex structure on $\C_I$
and $J$ orthogonal to $I$, by the splitting Lemma \ref{splitting}.
\end{proof}


Let us now discuss the case where $f$ is a slice function, i.e.,
induced by a stem function $F$.

\begin{proposition}\label{delta-stem}
  Let $\Omega_D$ be circular domain, arising as circularization
  of a symmetric domain $D$ in $\C$.
  
  Let $f:\Omega_D \to\H $ be a slice function
  which is induced by a stem function $F:D\to\H \tensor\mathbb{C}$.
  Then  $\partial_* f$, $\bar\partial_*f$ and $\Delta_*f$ are slice functions
  on $\Omega_D\setminus\R$ induced by the stem functions
  $\frac {\partial F}{\partial z}$, $\frac{\partial F}{\partial \bar z}$
  resp.~$\Delta F$.
\end{proposition}

\begin{proof}
  This may be deduced from the definitions of the respective operators
  using the representation formula for slice functions.
  See \cite{ghilperglobal},~Theorem 2.2.
\end{proof}

In particular, 
  for slice functions these operators $\partial_*$, $\bar\partial_*$
  and $\Delta_*$ are
  well-defined everywhere, including at the real points, whereas for
  arbitrary $C^1$-functions they are defined only outside $\R$.

\begin{corollary}
A slice function $f$ is annihilated by $\Delta_*$ if and only if its stem
function $F$ is harmonic.
\end{corollary}

\begin{corollary}
Let $f$ be a $C^2$ slice function on a circular domain $\Omega_D$.
If $\Delta_*f\equiv 0$, then $f$ is real-analytic.
\end{corollary}


\begin{corollary}
Let $f:\H\to\H$ be a $C^2$ slice function.
If $f$ is bounded and $\Delta_*f$ vanishes identically,
then $f$ is constant.
\end{corollary}

\begin{proof}
  $\Delta_*f\equiv 0$ implies the harmonicity of the stem function $F$.
  Now boundedness of $f$ implies boundedness of $F$ which leads to a
  contradiction unless $F$ (and therefore also $f$) is constant.
\end{proof}

\begin{theorem}\label{harm-star}
Let $\Omega_D$ be a slice domain.
Let $f:\Omega_D\to\H$ be a $C^2$ slice function.
Assume that $D$ is simply-connected.

\begin{enumerate}
\item
$\Delta_*f$ is vanishing identically if and only if $f$
can be written as a sum of a regular function $g$ and an
anti-regular function $h$.
\item  Assume $\Delta_*f\equiv 0$.
  Let $D_\R=D\cap\R$.

  Then $f$ can be written as a sum of 
a slice-preserving regular function $g$ and a slice-preserving
anti-regular function $h$ if and only if 
\[
f(x)\in\R,\ \forall x\in D_\R\quad\text{ and }(\partial_*f)(x)\in\R,
\ \forall x\in D_\R,
\]
which in turn holds if and only if
\[
f(x)\in\R,\ \forall x\in D_\R\quad\text{ and }(\bar\partial_*f)(x)\in\R,
\ \forall x\in D_\R,
\]
or if and only if

\[
\exists \,\, p\in D\cap\R:f(p)\in\R, 
(\bar\partial_*f)(x)\in\R, \ \forall x\in D_\R\quad\text{ and }
(\partial_*f)(x)\in\R,
\ \forall x\in D_\R.
\]

\end{enumerate}
\end{theorem}

\begin{proof}
Let $F$ be the stem function inducing $f$.
Then $\Delta_*f$ is a slice function induced by the stem function
$\Delta F$. 
Now $F$ is a map from $D$ to the complex vector space $\HC$.
Hence $\Delta F$ vanishes iff $F$ is harmonic iff $F=G+H$ for
a holomorphic function $G:D\to\HC$ and an anti-holomorphic
function $H:D\to\HC$.

We have to verify that $G,H$ may be taken to be stem functions. 
To state it more precisely:
We have to show: 
{\em If $F:D\to\HC$ is a map such that $\overline{F(\bar z)}
=F(z)$ and such that $F=G+H$ for a holomorphic map $G$ and an
anti-holomorphic map $H$, then $G$ and $H$ may be chosen in such a way
that $\overline{G(\bar z)}
=G(z)$ and $\overline{H(\bar z)}
=H(z)$.}
 Now $G(z)-\overline{G(\bar z)}$ is holomorphic,
$H(z)-\overline{H(\bar z)}$ is anti-holomorphic, and
\[
\left( G(z)-\overline{G(\bar z)} \right)+
\left( H(z)-\overline{H(\bar z)} \right)=0
\]
because of $F=G+H$ and $\overline{F(\bar z)}
=F(z)$.
It follows that there is a constant $c$ such that
\[
\left( G(z)-\overline{G(\bar z)} \right)=c=
-\left( H(z)-\overline{H(\bar z)} \right).
\]
We observe that $c$ is totally imaginary and that
\[
\overline{G(\bar z)-c/2}=\overline{G(\bar z)} +c/2=
G(z) - c/2
\]
Thus, by replacing $G$ with $G- c/2$
we may turn $G$ into a stem function.
Correspondingly we replace $H$ by $H+c/2$.

Finally we let $g$, resp.~$h$, be the slice functions
induced by the stem functions $G$, resp.~$H$, and observe that
$g$ is regular, because $G$ is holomorphic and $h$ is anti-regular,
because $H$ is anti-holomorphic.

Now we demonstrate $(2)$. 

First we observe that
\[
\partial_*f+\bar\partial_*f = \partial_1f
\]
Hence, if two of these values are zero, so is the third.

Next we claim:

{\em $f(\Omega_D\cap \R)\subset\R$ if and only
  if  $(\partial_1f)$ is real on $\Omega_D\cap\R$ and
  $\exists \,\, p\in \Omega_D\cap \R$ such that $f(p)\in\R $.}

The direction ``$\implies$'' is obvious.
To verify the converse, let $I$ denote a connected component ot
$\Omega_D\cap\R$. Then $f(I)\subset\R$. This implies that the stem function
$F$ maps $I$ into $\R\tensor\C\subset\H\tensor\C$.
Because $F:D\to\HC$ is a holomorphic map, we may now conclude
with the identity principle that $F(D)\subset \R\tensor\C$.
Using $\overline{F(\bar z)}=F(z)$, this implies that $F(\R\cap D)\subset \R$
which in turn (via representation formula) implies $f(\R)\subset\R$.

This completes the proof, because a slice function $f$ on a slice
domain $\Omega_D$ is slice-preserving if and only if the stem function
$F$ satisfies $F(D\cap\R)\subset\R\tensor_\R\R\subset \HC$.

\end{proof}

\begin{remark}
We observe that in \cite{perotti0} a function in the kernel of $\Delta_*$ is called {\em {slice-harmonic}}.
\end{remark}

\begin{definition}
\[
(\Delta'f)(q)=\left(\Delta_*\int_{\S}R_wfd\mu(w)\right) (q)
\]
We observe that $(\Delta' f)(q)=\frac{1}{2}\Delta_* (Tr(f))(q)=\frac{1}{2}\Delta_* (f+f^c)(q)$ for all $I \in \S .$ \\
Analogously we can introduce another second order operator in the following way:
$$
(\Delta'' f)(q)=(\Delta_* N(f))(q)=(\Delta_* (f \cdot f^c)) (q)
$$
for all $I \in \S .$ \\
On one side $\Delta '$ and $\Delta_*$
are linear operators, on the other hand, $\Delta''$ is not a linear operator.
\end{definition}
By the way, for $\Delta ''$ a product formula holds:

  \begin{proposition}\label{product-rule}
   Let $f,g$ be slice functions.
   Then
   \[
   \Delta''(f*g)=(f^s)*\Delta'' g + (\Delta'' f) * (g^s)
   +\left( \partial_*(f^s)* \bar\partial_*(g^s) \right)
   +\left( \bar\partial_*(f^s)* \partial_*(g^s) \right)
    \]
  \end{proposition}
  \begin{proof}
    First we note that
    \[
    \Delta''(f*g)=\Delta_*(N(f*g))=\Delta_*((f*g)^s)=
    \Delta_*(f*g*g^c *f^c)
    \]
    Next we recall that (with respect to the slice product $*$) a slice
      preserving function commutes with every other slice function.
      Since $g*g^c$ is always slice preserving, we obtain:
      \begin{align*}
        &\Delta_*(f*g*g^c *f^c)=\Delta_*(f*f^{c}*g*g^c)\\
        =&
      \Delta_*(f^s* g^s)=\partial_*\bar\partial_*(f^s\cdot g^s)\\
      =& \partial_* \left( (\bar\partial_*f^s)\cdot g^s
      +   f^s\cdot(\bar\partial_*g^s) \right)\\
      =& (\partial_* \bar\partial_*f^s)\cdot g^s
      + (\bar\partial_*f^s)\cdot (\partial_* g^s)
      +  (\partial_*f^s)\cdot(\bar\partial_*g^s)
      +       f^s\cdot(\partial_*\bar\partial_*g^s) \\
      =& f^s*(\Delta'' g) + (\Delta'' f) * g^s
      +      (\bar\partial_*f^s)\cdot (\partial_* g^s)
      +(\partial_*f^s)\cdot(\bar\partial_*g^s)\\
      \end{align*}
  \end{proof}

\begin{remark}
In \cite{CGCS} the following global first order differential operator
was introduced:
$$
G(x)=(x_1^2 +x_2^2 + x_3^2) \frac{\partial}{\partial x_0} +(x_1 i +x_2 j + x_3 k) \cdot \sum_{j=1}^3 x_j \frac{\partial}{\partial x_j}
$$
where $x=x_0 + x_1 i + x_2 j + x_3 k.$
Direct calculations verify readily that $G=y^2\bar\partial_*$ where
\[
y=|\Im(x)|=
\sqrt{\sum_{k=1}^3|x_i|^2}.
\] 
Whereas the operator $\Delta_*$ is defined everywhere only if applied
to slice functions, $G$ is everywhere defined for any $C^1$-function.

Still, we believe that the additional factor $y^2$ (which guarantees the
applicability of the operator even to non-slice functions) is
somewhat unnatural.

In particular, if we try to construct a second order differential operator,
we see that
\[
\bar G G=y^4\Delta_*  -2Iy^3\bar\partial_*.
\]
Thus $\bar G G$ even applied to slice functions on a complex slice
$\C_I$ is not merely the multiple of the ordinary complex
Laplacian.
Moreover, $\bar GG$ annihilates regular functions, but not
anti-regular functions. \\

We observe that in \cite{ghilperglobal} a global operator related to $G$ was studied.
\end{remark}

\begin{remark}
It may happen that a function $f:\H\to\H$ satisfies both
$\bar\partial_*Tr(f)=0$ and $\bar\partial_* N(f)=0$, but 
nevertheless is not
regular. For example, let $I,J\in\S$ be orthogonal
(i.e.~$IJ=-JI$), and consider the function 
\[
f:q\mapsto I\cos(\Re q)+J\sin(\Re q).
\]
This is a slice function (induced by the stem function
$z\mapsto I\tensor\cos(\Re z) + J\tensor \sin(\Re(z))$.)
It is not regular because it is not open,
although both $Tr(f)=f+f^c=0$ and $N(f)=ff^c=1$ are
regular (in fact constant) functions.
\end{remark}

\begin{definition}
A function $f:\H\to\H$ is {\em rotationally equivariant}
if $R_I(f)=f$ for all $I\in\S$.%
\footnote{see definition~\ref{defRw} for the notion $R_I(f)$.}
\end{definition}

\begin{lemma}\label{LemRotEq}
For a function $f:\H\to\H$, the following conditions are equivalent:
\begin{enumerate}
\item
$f$ is rotationally equivariant.
\item
$R_I(f)=f$ for all $I\in\S$.
\item
$S_I(f(q))=f(S_I(q))$ for 
all $I\in\S$ and $q\in\H$.
\item
If $g:q\to g\cdot q$ denotes the $\R$-linear action of 
$SO(3,\R)$ on $\H$ which is trivial on $\R$ and which is the natural
orthogonal transformations on $\left<I,J,K\right>=\R^3$,
then
$f:\H\to\H$ is equivariant for this action, i.e., $g\cdot f(q)=f(g\cdot q)$
for all $g\in SO(3,\R)$, $q\in \H$.
\item
$f$ is induced by a stem function $F:\C\to \H\tensor\C$ with
$F(\C)\subset \R\tensor\C$,
i.e., both $F_1$ and $F_2$ are real-valued for $F=F_1+ F_2\imath$.
\item
$f$ is induced by a stem function $F$ and $f$ is slice-preserving,
i.e., $f(\C_I)\subset\C_I$, $\forall I\in\S$.
\end{enumerate}
\end{lemma}

\begin{proof}
The equivalence $(5)\iff(6)$ is well-known.
$(1)\iff(2)\iff(3)$ follow directly form the respective definitions.
Lemma~\ref{generate-so} implies $(3)\iff(4)$.
$(5)\Rightarrow(2)$ is easy.
Finally, assume $(3)$. Fix $I\in\S$. Then $(3)$ implies
$S_I(f(q))=f(S_I(q))$. For $q\in \C_I$ we have $S_I(q)=q$ and therefore
$S_I(f(q))=f(q)$, which implies $f(q)\in\C_I$. Thus $f(\C_I)\subset \C_I$.
We define functions $g,h:\C\to\R$ such that
\[
f(x+yI)=g(x+yi)+h(x+yi)I\ \ (x,y\in\R)
\]
Now, for every $J\in\S$, there is an orthogonal transformation $\phi$
fixing $\R$ with $\phi(I)=J$. Using $(4)$, we have
\begin{align*}
f(x+yJ) &=f(\phi(x+yI))=\phi( f(x+yI))\\
        &=\phi\left(
g(x+yi)+h(x+yi)I\right)= g(x+yi)+h(x+yi)J
\end{align*}

Thus $f:\H\to\H$ is induced by the stem function $F$ given as
$F(x+yi)=g(x,y) \tensor 1 + h(x,y)\tensor i$.
(The fact $\overline{F(z)}=F(\bar z)$ 
follows from Lemma~\ref{stem-basic}.)
This completes the proof of $(3)\Rightarrow(5)$.
\end{proof}

\begin{lemma}\label{stem-real-fct}
Let $h:\H\to\R$. 
Then the following conditions are equivalent:
\begin{enumerate}
\item
$h$ is induced by a stem function.
\item
$h$ obeys the representation formula.
\item
$h$ is rotationally equivariant.
\item
$h$ is rotationally invariant.
\end{enumerate}
Moreover, in this case the stem function $H$ 
can be defined as 
\[
H(x+yi)=h(x+yI)\tensor 1
\]
for all $x,y\in\R$, $I\in\S$.
\end{lemma}

\begin{proof}
Assume that $h$ is induced by a stem function $H$.
The formula
\[
h(x+yI)=F_1(x+yi) + I F_2(x+yi)\ \forall x,y\in\R, I\in\S
\]
combined with $h(q)\in\R,\ \forall q\in\H$, implies that $F_2$ vanishes
and that $F_1(\C)\subset\R$.
Then $h$ is rotationally equivariant (by Lemma~\ref{LemRotEq} $(5)\Rightarrow(1)$)
and rotationally invariant (by Lemma~\ref{LemRotInv} $(5)\Rightarrow(1)$).

Conversely, by Lemma~\ref{LemRotEq} $(1)\Rightarrow(5)$,
(resp.~ by Lemma~\ref{LemRotInv} $(1)\Rightarrow(5)$), imply that $h$ admits
a stem function if $h$ is rotationally equivariant, resp.~rotationally
invariant.
\end{proof}

\begin{proposition}\label{real-part}
Let $h:\H \to\mathbb{R}$ be a 
slice
function with $\Delta'h=0$.

Then
there exists a slice-preserving
regular function $f$ such that $h=\Re  \, (f)$.
\end{proposition}

\begin{proof}
Since $h$ is real-valued and a slice function, it is also
rotationally equivariant and rotationally invariant
(due to Lemma~\ref{stem-real-fct}).
Being rotationally equivariant implies $\Delta'h=\Delta_*h$.
Being rotationally invariant implies
\[
h(x+yI)=h(x-yI) \,\, \forall x,y\in\R,\ \forall
I\in\S
\]
which in turn implies that
\[
\partial_I h
\]
vanishes on the real line.
It follows that $(\partial_*f)(x)=\frac 12 (\partial_1 f)(x)\in\R$
for all $x\in\R$.
Hence all the assumptions of Theorem~\ref{harm-star} (2)
are fulfilled, and we may deduce that $h=g+k$ where 
$g$ is regular, $k$ is anti-regular and both $g$ and $k$
are slice-preserving.
The condition $h(q)\in\R$ is equivalent to $h(q)-\overline{h(q)}=0$.
Therefore
\[
g(q)-\overline{k(q)}=\overline{g(q)}-k(q)\,\, \forall q\in\H.
\]
The left hand side of this equation is a regular function, while the
right hand side is anti-regular. Thus both must be constant.
That the right hand side is constant, implies:
\[
k(q)= \overline{g(q)}-\overline{g(0)}+k(0)\,\, \forall q\in\H.
\]
Hence
\[
h(q)=g(q)+k( q)=g(q)+\overline{g(q)}-\overline{g(0)}+k(0)
=2 \Re\left(g(q)\right)-\overline{g(0)}+k(0)
\ \ \forall q\in\H.
\]
Finally, $h(\H)\subset\R$ implies $-\overline{g(0)}+k(0)\in\R$.
This proves the assertion for 
\[
f(q)=2 g(q)-\overline{g(0)}+k(0).
\]
\end{proof}

\begin{proposition}
Let $u$ be real-valued and 
rotationally invariant (in the sense of Definition \ref{def-rot})
$C^2$-function.
Let $f$ be regular.

Then
\[
\Delta'(u\circ f)=\int_{\S}
\left( (\Delta_* u)\circ (R_wf)\right)
|(\partial_*f)(S_wq)|^2 d\mu(w)\quad 
\]
\end{proposition}

 \begin{proof}
  Note that a real valued rotationally invariant function $u:\H\to\R$ is
  automatically a slice function. (Lemma \ref{stem-real-fct}.)
  Hence $\Delta'u$ is well defined.

  The claim of the proposition
  is a consequence of the complex computation \\ 
$\Delta (u \circ f)= \left( (\Delta u)\circ f\right)
|(\partial_*f)(z)|^2$ 
(applied on $\C_I$) and of the definition of $\Delta '$. For more details, see the proof of the next proposition.
\end{proof}

\begin{proposition}\label{armcomp}
Let $u \colon \H \to\mathbb{R}$ be a rotationally invariant
$C^2$-function with $\Delta'u=0$.
Let $f \colon \H  \to \H  $ be a slice preserving regular
function, then $u \circ f$ is such that $\Delta' (u\circ f)=0.$
\end{proposition}
\begin{proof}
  First observe that  $u$ is a slice function,
  because it is real valued and rotationally
  invariant (Lemma \ref{stem-real-fct}).
Next we prove that, under the hypotheses of the proposition,
\begin{equation}\label{invrot}
R_w (u \circ f)= u \circ (R_w (f)) \,\,\, \forall \,\, w \in \S .
\end{equation}
Indeed $\forall \,\, w \in \S $: 
\begin{align*}
 R_w ( u\circ f) &= w(u(f(w^{-1}qw)))w^{-1} = w(u(R_w (f))w^{-1} \\
&= S_w ^{-1} (u \circ R_w (f)) = u ( R_w (f)).
\end{align*}
Then, by definition, $\Delta' (u \circ f) =\Delta_* \int\limits_{w \in \S } R_w (u \circ f) d \mu.$ \\
By \eqref{invrot}, 
$$\Delta_* \int\limits_{\S } R_w(u \circ f) d \mu(w)=\Delta_* \int\limits_{\S } u \circ R_w(f) d \mu(w)=  \int\limits_{ \S } (\Delta_* (u \circ R_w f) d \mu(w) =$$
$$=\int\limits_{\S }
{(\Delta_* u)\circ (R_w(f))}{|(R_w(f))'|^2} d \mu(w)=\int\limits_{\S }
{(\Delta_* u)\circ (R_w(f))}{|(f)' (S_w(q))|^2} d \mu(w)=0$$  
where $g' : =\partial_* g$.
\end{proof}

\begin{proposition}\label{isol-zero}
Let $u \colon \H  \to \mathbb{R}$ be 
a $C^2$-function such that $\Delta' u\equiv 0 $ on $\H\setminus\R$. 
Then 
$u$ admits no real isolated zero.
\end{proposition}

\begin{remark}
  At a real point $\Delta'u$ is defined only if $u$ is a slice function.
  But in the proposition $\Delta'u$ is only considered for points
  outside $\R$. Hence one does not need to require $u$ to be
  a slice function.
\end{remark}

\begin{proof}
Assume the contrary, i.e., assume that
$u$ has an isolated zero in a point $a\in\R$. 
Then there exists an open neighborhood
$W$ of $a$ such that $u$ has no zero on $W\setminus\{a\}$.
For dimension reasons, $W\setminus\{a\}$ is connected. Thus $u$ is either 
everywhere positive or everywhere negative on $W\setminus\{a\}$.
Without loss of generality, assume that $u>0$ on $W\setminus\{a\}$.

Define
\[
\tilde u(q)=\int_\S \left(R_wu)\right(q) d\mu(w).
\]
For $q$ sufficiently close to $a$ (but $q\ne a$) we have
$\S_q\subset W\setminus\{a\}$. For such $q$ we
have 
\[
(R_wu)(q)>0\ \forall w\in\S
\]
and therefore $\tilde u(q)>0$.
On the other hand, $\tilde u(a)=u(a)=0$, because $a\in\R$.
Thus $\tilde u$ has a strict local minimum in $a$. 

By construction  $\Delta'u=\Delta_*\tilde u$
on $\H\setminus\R$. 
Fix $I\in\S$.
By definiton, on $\C_I$ the operator $\Delta_*$ agrees
with the ordinary complex Laplacian, if we identify $\C_I\simeq\C$ as usual.
Thus $\tilde u$ restricts to a $C^2$-function on $\C_I$ which
is harmonic on $\C_I\setminus\R$.
By continuity of $\Delta\tilde u$, the function $\tilde u$ is
harmonic on the whole of $\C_I$.
Thus we obtain a harmonic function on $\C_I\simeq\C$ with a strict local
minimum in $a$.
This is impossible.
\end{proof}

\section{A kind of Poisson formula}
\subsection{$\H$-valued measures}
\begin{theorem}\label{H-measure}
Let $p\in\H $ and let $S\subset\H $ be a $3$-dimensional 
sphere such that $p$ is in its interior.
Let $\Omega$ denote a circular domain containing both $S$ and its
interior.

Then there exists an $\H $-valued measure $\mu$ on $S$ which is absolutely
continuous with respect to the euclidean measure such that
\[
f(p)=\int_{S} f(q)d\mu(q)
\]
for every regular function $f$ defined on $\Omega$.
\end{theorem}

\begin{proof}
We first discuss the special case where $p\in\mathbb{R}$. In this case for
each $I\in\S$ the restriction of $f$ to $\mathbb{C}_I$ is holomorphic,
and $\mathbb{C}_I$ and $S$ intersect in a $1$-dimensional sphere which
contains $p$ in its interior. We thus may construct $d\mu$ first
taking the measure on $S\cap\mathbb{C}_I$ given by the classical Poisson
formula, and then averaging over $I\in\S$ with respect to any
probability measure of $\mathbb{S}.$

Now assume $p\not\in\mathbb{R}$. Fix $I\in\S$ such that $p\in\mathbb{C}_I$.
Let $c=s+vJ$ ($s,v\in\mathbb{R}, v>0, J\in\S$)
denote the center of the sphere $S$ and $\rho$ its radius.
Define $\bar G=\{t+yi\in\mathbb{C}: t,y\in\mathbb{R}, y\ge 0, \exists \,\, H\in\S: t+yH\in S\}$.

Then
\[
\bar G = \{ (t+yi): \exists \,\, H\in\S : (t-s)^2 + |yH-vJ|^2 = \rho^2 \}.
\]
We observe that $\S$ is connected and 
that $H\mapsto |yH-vJ|^2$ (for fixed $y,v>0,J\in\S$) defines
a continuous map which evidently takes its maximum in $H=-J$ 
(with $(y+v)^2$ as its value) and 
its minimum in $H=J$ (with value $|y-v|^2$).

From this we may deduce :
\[
\bar G = \{ (t+yi): |t-s|\le\rho,\quad 
|y-v| \le \sqrt{ \rho^2 -|t-s|^2} \le |y+v|
\}.
\]

Let us now fix $t\in\mathbb{R}$ and investigate for which $y>0$ we
have $t+iy\in\bar G$. We define $K=\sqrt{ \rho^2 -|t-s|^2}$
and 
obtain the following condition:
\begin{align*}
&|y-v| \le \sqrt{ \rho^2 -|t-s|^2}=K \le |y+v|\\
\iff & |y-v| \le  K \le |y+v|\\
\iff & v-K\le y \le v+K \text{ and }-v+K\le y\\
\iff & |v-K| \le y \le v+K.
\end{align*}

It follows that the interior $G$ of $\bar G$ is simply-connected and
therefore biholomorphic to the unit disc.

  Let $\tilde p=x+yi\in\C$ such that $x,y\in\R$, $y\ge 0$ and
  $p=x+yH$ for some $H\in\S$.
  Then $\tilde p$ is in the interior of $G$.

We fix such a biholomorphic map $\psi:G\to \Delta$ 
and recall
that it extends
continuously to the respective boundaries.
We may and do require $\psi(\tilde p)=0$.

By the classical mean value theorem 
\[
F(0)=\int_0^1\int_0^{2\pi} F(re^{i\theta}) \frac {d\theta}{2\pi} d\sigma
\]
for every holomorphic function $F:\mathbb{C}\to\mathbb{C}$, every $r>0$ and
every probability measure $\sigma$ on $[0,1]$.

Pulling-back with $\psi$ yields a probability measure $d\xi$ on $G$
such that
\begin{equation}\label{eq-6-1}
F(p)=\int_G F(w)d\xi(w)
\end{equation}
for every holomorphic function $F$. The measure $d\xi$ constructed
in this way is
absolutely continuous, if the measure $\sigma$ on $[0,1]$ used in the
construction is taken to be absolutely continuous.

For each point $t+yi\in G$ we have a $2$-sphere $t+y\S$ in $\H $.
Let $V$ denote the ``imaginary subspace'' of $\H$,
  i.e., the $\R-$vector
subspace of $\H$ generate by $y\S$.
The intersection of the $3$-sphere $S$ with
real affine subspace $t+V$ is a sphere
(of dimension $\le 2$). Thus $S\cap(t+y\S)$ is an intersection
of two spheres in a three-dimensional real affine space and
therefore again a sphere.

We let $\eta$ denote the involution defined by sending each 
element of $\Sigma_{t,y}=S\cap (t+y\S)$ to its antipodal element.

Since $\Sigma_{t,y}\subset (t+y\S)$, for every
$q\in\Sigma_{t,y}$ there are $J,H\in\S$ such that
$q=t+yJ$ and $\eta(q)=t+yH$.

By the generalized representation formula
(Proposition~\ref{gen-repr})
we have
\[
f(t+yI)= M_1(J,H) f(q) + M_2(J,H) f(\eta(q)) \quad \forall q\in\Sigma_{t,y}
\]
for every regular function $f$.

With $m_1(q):=M_1(J,H)$ and $m_2(q)=M_2(J,H)$ we
obtain continuous functions
$m_i:\Sigma_{t,y} \to\H ,$ for $i=1,2,$ such that 
\[
f(t+yI)= m_1(q) f(q) + m_2(q) f(\eta(q)) \quad \forall q\in\Sigma_{t,y}
\]
for every regular function $f$.

In particular
\[
f(t+yI)= \int_{q\in \Sigma_{t,y}}  m_1(q) f(q) + m_2(q) f(\eta(q)) d\alpha(q)
\]
for every probability measure $\alpha$ on $\Sigma_{t,y}$.
Hence we may choose an absolutely continuous probability measure
$\beta_{t,y}$ 
on $\Sigma_{t,y}$
such that
\begin{equation}\label{eq-6-2}
f(t+yI)= \int_{\Sigma_{t,y}}  f(q) d\beta(q)\quad \forall f
\end{equation}
We recall that regular functions restrict to 
holomorphic ones on $\mathbb{C}_I$.

We may therefore combine the above constructions 
(see equations~\eqref{eq-6-1},\eqref{eq-6-2}) to obtain
\[
f(p)=\int_{t+yi\in G} \int_{q\in \Sigma_{t,y}}
f(q)
d\beta(q) d\xi(t+yi)
\]
(with $\Sigma_{t,y}=S\cap(t+y\S)$).
\end{proof}

\subsection{Poisson's formula}

\begin{proposition}[Poisson's Formula]
  Let $\mu$ denote a probability measure on $\S$.
  Let $u \colon \overline{\mathbb{B}_R} \to \mathbb{R}$ be
  a rotationally invariant $C^2$- function.
  Assume  that $\Delta'u\equiv 0$.
  \footnote{$\Delta'u$ is defined on $\R$ for slice functions only.
    This is no problem, because every rotationally invariant function
    is a slice function.}
  Let $a \in \mathbb{R}$.
  Then the following formula holds:
$$
  u(a)=\frac{1}{2\pi} \int\limits_{\S }
  \int\limits_{0}^{2\pi}
  \frac{R^2-a^2}{|R \cdot e^{I\theta}-a|^2}u(R \cdot e^{I\theta})d\theta d\mu(I)
$$
\end{proposition}
\begin{proof}
  Due to Lemma~\ref{stem-real-fct} the function $u$ is a {\em slice function}.
  Thus we may conclude from Proposition~\ref{real-part} that $u$ is the
  real part of a slice-preserving regular function $f$. Therefore
  for each $I\in\S$, the restriction of $u$ to $\C_I$ is the real part
  of a holomorphic function from $\C_I$ to $\C_I$ and the above formula
  follows from the complex Poisson formula.
\end{proof}
\begin{remark}
As in Proposition \ref{GMVF}, it is possible to generalize the above Poisson Formula at any $a \in \mathbb{H}$ with an integration on the circularization of $\partial \Delta(a,r) \cup \partial\Delta(\overline{a},r)$ instead of an integration on $\partial \mathbb{B}_R.$ 
\end{remark}

\section{A Jensen's Formula}
The goal of this section is to prove a quaternionic version of Jensen's formula.
For this purpose we need Blaschke factors.

\subsection{Quaternionic $\rho$-Blaschke factors}

In this subsection we are going to reproduce some results proved in \cite{irenebla,shur} for a modification of quaternionic Blaschke factors.

\begin{definition}
Given $\rho>0$ and $a\in\H $ such
that $|a|<\rho$. We define the $\rho$\textit{-Blaschke factor} at $a$ 
as the following semi-regular function on $\mathbb{H}$:
\begin{equation*}
B_{a,\rho}:\H \rightarrow\hat{\H },\quad
B_{a,\rho}(x):=(\rho^{2}-x\bar a)*(\rho(x-a))^{-*}.
\end{equation*}
\end{definition}

We observe that
\begin{align*}
B_{a,\rho}
&= (\rho^2-q\bar a)*(q-\bar a)
\left(\rho(q^2-q(a+\bar a)+|a|^2)\right)^{-*}\\
&= (q^2(-\bar a) +q(\rho^2+\bar a ^2) -\rho^2\bar a)\left(\rho(q^2-q(a+\bar a)+|a|^2)\right)^{-1}
\end{align*}
(using that $g(q)^{-*}=(g(q))^{-1}$ for any slice-preserving
function $g$, hence in particular for $g(q)=\rho(q^2-q(a+\bar a)+|a|^2)$).

In particular,
\[
|B_{a,\rho}(0)|=\left| \frac{\rho}{a}\right|, \,\, \textrm{if} \,\, a \neq 0 
\]
and
\[
|B_{a,\rho}^{-*}(0)|=\left| \frac{a}{\rho}\right|.
\]

\begin{remark}\label{zerobla}
We observe that 
$(B_{a,\rho})^{-*}$ has a zero of multiplicity one at $a$
and no other zeros or poles in $\B_\rho$.
\end{remark}

The following is a consequence of Theorem 5.5 of \cite{irenebla}.

\begin{theorem}\label{borderbla}
Given $\rho>0$  and $a\in\H $. The
$\rho$-Blaschke factors $B_{a,\rho}$ have the following properties:
\begin{itemize}
\item they satisfy $B_{a,\rho}(\H \setminus\mathbb{B}_{\rho})\subset \B_1$ and $B_{a,\rho}(\mathbb{B}_{\rho})\subset \H \setminus\B_1$.
\item they send the boundary of the ball  
$\partial\mathbb{B}_{\rho}$ in the boundary of the ball $\partial \B_1$.
\end{itemize}
\end{theorem}

\subsection{Jensen's formula}
First we prove a variant of Jensen's formula for the special
case where there are neither zeros nor poles.

\begin{proposition}\label{jensen-ni-ni}
Let $\rho>0$ and let $f$ be a regular function defined in a
neigbourhood of $\overline{\mathbb{B}_{\rho}}$. Assume that $f$ has no
zeros in $\overline{\mathbb{B}_{\rho}}$.

Let $\mu$ be a probability measure on $\S$.

Then
\[
\log|f(0)| \le 
\displaystyle\frac{1}{2\pi}\displaystyle \int_0^{2\pi} \int_{\S }\log|f(\rho\cos\theta +\rho\sin\theta I)|d\mu(I)d\theta 
\]
with equality if $f$ is slice-preserving.
\end{proposition}

\begin{proof}
Fix an imaginary unit $I$. 
Choose another imaginary unit $J$ such that
$IJ=-JI$
(i.e., $I$ and $J$ are supposed to be orthogonal). 
Thus, using the ``Splitting Lemma'' \ref{splitting},
there are two holomorphic functions $g,h$ with values in $\C_I$
defined on a neighborhood
of $\bar\Delta_\rho=\{z\in\C_I:|z|\le\rho\}$ such that
\[
f(q)=g(q)+h(q)J,\ \ \forall q\in\Delta_\rho=\B_\rho\cap\C_I.
\]
Then $|f(q)|^2=|g(q)|^2+|h(q)|^2$.
Now, $\log\left( |g|^2+|h|^2\right)$ is subharmonic
for any two holomorphic functions $g,h:\C\to\C$. Thus we have
subharmonicity of $\log|f|^2$ and consequently
\[
\log|f(0)|\le \frac{1}{2\pi}\int_0^{2\pi}
\log|f(\rho e^{It})|dt
\]

Finally, by integration over the sphere of imaginary units we obtain
the assertion.
\end{proof}

In order to deal with the general case (where the function $f$ is
allowed to have zeros or poles) we need some preparations.

\begin{lemma}
Let $f,g$ be regular functions on an open neighborhood of
$\partial \B_\rho=\{q\in\H:|q|=\rho\}$.

Assume that $|g(q)|=1$ for all $q\in\partial \B_\rho$.

Then
$|f(q)|=|(f*g)(q)|$ and $|g^{-*}(q)|=1$ for all
$q\in\partial \B_\rho$.
\end{lemma}

\begin{proof}
The formula
\[
|p^{-1}qp|= |q|,\ \ \forall p\in\H^*,q\in\H
\]
implies that
$ f(q)^{-1}qf(q)\in\partial\B_\rho$ whenever $q\in\partial\B_\rho$.
Combined with
\[
(f*g)(q)=f(q)g\left( f(q)^{-1}qf(q) \right)\quad
\text{for $q$ with $f(q)\ne 0$}
\]
and $|g(q)|=1\ \forall q\in\partial\B_\rho$ we obtain
\[
|(f*g)(q)|=|f(q)|\ \ \forall q\in\partial\B_\rho.
\]
If we apply this equation to $f=g^{-*}$, we obtain
\[
1=|(g^{-*}*g)(q)|=|(g^{-*})(q)|\ \ \forall q\in\partial\B_\rho.
\]
\end{proof}

\begin{proposition}\label{factorization}
Let $f$ be a semi-regular function on a neighborhood of
$\bar \B_\rho$, with neither zeros nor poles on $\partial \B_{\rho}$.

Then there exist ``$\rho$-Blaschke factors'' $B_1,\ldots,B_r$
and a regular function without zeros $f_0$
such
that:
\[
f=f_0*B_1*\ldots *B_r.
\]
\end{proposition}

Here a function $B$ is called $\rho$-Blaschke factor, if $B=B_{a,\rho}$
or $B=B_{a,\rho}^{-*}$ for an element $a\in \B_\rho$.

\begin{proof}
Let $g,h$ be regular functions such that $f=g^{-*}*h$.
First we claim that there exist $\rho$-Blaschke factors $B_1,\ldots, B_s$
and a regular function $\tilde h$ without zeros
such that $f=g^{-*}*\tilde h*B_1*\ldots* B_s$ .
We proceed recursively. If $h$ admits a zero
in a point $a\in\B_\rho$, then there exists
an element $b\in\S_a$ such that $h=h_0*(q-b)$. Recall
that
\[
B_{b,\rho}=(\rho^2-q\bar b)*(\rho(q-b))^{-*}
= (\rho(q-b))^{-*} *(\rho^2-q\bar b)
\]
and therefore
\[
(q-b)=\frac 1{\rho}\left( \rho^2-q\bar b\right)
* B_{b,\rho}^{-*}.
\]
Thus
$f=g^{-*}*h_1*B_{b,\rho}^{-*}$
with
\[ 
h_1(q)=h_0(q)\frac 1{\rho} * \left( \rho^2-q\bar b\right)
\]
being regular.
Repeating this procedure recursively, we obtain a regular function
$\tilde h$ without zeros in $\bar\B_\rho$
 and $\rho$-Blaschke factors $B_1,\ldots, B_s$
such that
\[
f=g^{-*}*\tilde h*B_1*\ldots* B_s
\]
Define
\[
f_1=g^c*(\tilde h^{-*})^c
\]
Then 
$f_1$ is regular and 
\[
f=(f_1^{-*})^c*B_1*\ldots* B_s
\]
Repeating the above process, we obtain a regular function $\phi$
without zeros and $\rho$-Blaschke factors $B'_1,\ldots B'_t$
such that
\[
f_1=\phi* (B'_1*\ldots * B'_t)
\]
Consequently
\begin{align*}
f &=\left(\left( \phi* (B'_1*\ldots * B'_t) \right)^{-*}\right)^c
*B_1*\ldots* B_s \\
&
=(\phi^{-*})^c * \left( ((B'_1)^{-*})^c\right) * \ldots  (((B'_t)^{-*})^c)
* \left( B_1*\ldots* B_s \right)
\end{align*}
\end{proof}

\begin{proposition}[Jensen's Formula]\label{jensen}
Let $\Omega=\Omega_D$ be a circular
domain of $\H $ and 
let $f:\Omega\rightarrow\H \cup\{ \infty \}$ be a
semi-regular function.
Let $\rho>0$ such that the ball $\overline{\mathbb{B}_{\rho}}\subset\Omega$, $f(0)\neq 0,\infty$ and such that $f(y)\neq 0, \infty$, for any $y\in\partial \mathbb{B}_{\rho}$. \\
Assume  that (for the function $f$)
\begin{itemize}
\item $\{r_{k}\}_{k=1,2,..}$ are the punctual zeros, \\
\item $\{ w_n \}_{n=1,2,...}$ are the punctual real poles, \\
\item $\{\S _{a_{i}}\}_{i=1,2,..}$ are the spherical zeros, \\
\item $\{\S _{b_{j}}\}_{j=1,2,..}$ are the spherical poles, \\

\end{itemize}
everything repeated accordingly to their multiplicity.
Let $\mu$ be a probability measure on $\S$.

Then: \\
\begin{align*}
\log|f(0)| 
\le &
\displaystyle\frac{1}{2\pi }\displaystyle \int_0^{2\pi} \int_{\S }\log|f(\rho\cos\theta +\rho\sin\theta I)|d\mu(I)d\theta \,  \\
& -\displaystyle\sum_{|r_{k}|<\rho}\left(\log\displaystyle\frac{\rho}{|r_{k}|}\right) 
+ \displaystyle\sum_{|w_{n}|<\rho}\left(\log\displaystyle\frac{\rho}{|w_{n}|}\right)\\ 
&- 2\displaystyle\sum_{|a_{j}|<\rho}\left(\log\displaystyle\frac{\rho}{|a_{j}|}\right)+ 2\displaystyle\sum_{|b_{j}|<\rho}\left(\log\displaystyle\frac{\rho}{|b_{j}|} \right).
\end{align*}
\end{proposition}
Using the language of divisors as explained in section
\ref{sect-div}
we may reformulate this as follows:

\begin{proposition}[Jensen's Formula]
Let $\Omega=\Omega_D$ be a circular
domain of $\H $ and 
let $f:\Omega\rightarrow\H \cup\{ \infty \}$ be a
semi-regular function.
Let $\rho>0$ such that the ball 
$\overline{\mathbb{B}_{\rho}}\subset\Omega$, $f(0)\neq 0,\infty$ 
and such that $f(y)\neq 0, \infty$, for any $y\in\partial \mathbb{B}_{\rho}$.
Let $\mu$ be a probability measure on $\S$.

Then: \\
\begin{equation}\label{jensenNreg}
\begin{array}{ccc}
\log|f(0)| 
&\le 
\displaystyle\frac{1}{2\pi }\displaystyle \int_0^{2\pi} \int_{\S }\log|f(\rho\cos\theta +\rho\sin\theta I)|d\mu(I)d\theta \, + \\
&\\
& \displaystyle-\sum_{|p_k|<\rho} m_k\log \frac{\rho}{|p_k|}
\quad 
 \\
\end{array}
\end{equation}
for $div(f)=\sum m_k\{p_k\}$.
\end{proposition}

\begin{proof}
If $f$ has neither zeros, nor poles, this is
Proposition~\ref{jensen-ni-ni}.

For the general case, we
consider
 $div(f)=\sum_k m_k\{p_k\}$.
First of all, since $\overline{\mathbb{B}_{\rho}}\subset\Omega$, 
it follows from Corollary~\ref{cor-id} that  
there are only finitely many $k$ with $m_k\ne 0$,
hence the sums in the statement are finite
and $\sum_k |m_k|<\infty$.

From Proposition~\ref{factorization} we deduce that $f$ may be represented
in the form
\[
f=f_0*B_1*\ldots *B_r
\]
where $f_0$ is regular on a neighborhood of $\bar\B_\rho$
with neither zeros nor poles and each $B_i$ equals 
$B_{p_i,\rho}^{\epsilon_i *}$ 
for some $p_i\in\B_\rho$ and $ \epsilon_i\in\{+1,-1\}$.

Now 
\[
\log|f(0)|=\log|f_0(0)| +\sum_i \log |B_i(0)|=
\log|f_0(0)| -\sum_i \epsilon_i\log \frac{\rho}{|p_i|} 
\]
On the other hand
\[
|f(q)|=|f_0(q)| \ \ \forall q\in\partial \B_\rho
\]
because $|B_i(q)|=1$ for all $i\in\{1,\ldots,r\}$ and all $q\in\partial \B_\rho$.

Furthermore,
\[
\log|f_0(0)|  \le 
\displaystyle\frac{1}{2\pi}\displaystyle \int_0^{2\pi} \int_{\S }\log|f_0(\rho\cos\theta +\rho\sin\theta I)|d\mu(I)d\theta 
\]
(with equality if $f$ is slice-preserving)
due to Proposition~\ref{jensen-ni-ni}.

Combining these facts, we obtain the assertion.
\end{proof}

\begin{remark}
For every semi-regular function $f,$ its symmetrization
$f^s=N(f)=f*f^c$ is slice-preserving.
Therefore:
\begin{equation}
\begin{array}{ccc}
\log|f^s(0)| 
&=
\displaystyle\frac{1}{2\pi }\displaystyle \int_0^{2\pi} \int_{\S }\log|f^s(\rho\cos\theta +\rho\sin\theta I)|d\mu(I)d\theta \, + \\
& -2{\displaystyle\sum_{|p_k|<\rho} m_k\log \frac{\rho}{|p_k|} }
\quad 
 \\
\end{array}
\end{equation}
for $div(f)=\sum m_k\{p_k\}$.

However, there is no similar formula for $Tr(f)=f+f^c$, because
$div(f+f^c)$ is not completely determined by $div(f)$ whereas
$div(f^s)=2\cdot div(f)$.
\end{remark}

\begin{definition}
Let $f$ be a slice regular function on $\mathbb{B}_R.$ For all $0 < r < R$ we define:
$$
M_f (r) = \sup\limits_{|q|=r} |f(q)|  \,\,;
$$
$$
P_f (r) = \textrm{number of punctual zeros of f with multiplicities} \,\, ;
$$
$$
S_f (r ) = \textrm{number of spherical zeros of f with multiplicities};
$$
\vskip0.3cm
$$
n_f(r)=P_f(r)+2S_f(r).
$$
\end{definition}

Then 
\[
n_f(r)=\sum_{|a_k|\le r} m_k
\]
for a regular function $f$ with divisor $div(f)=\sum m_k\{a_k\}$.

\begin{corollary}
Let $f$ be a slice regular function defined in a neighborhood of $\overline{\mathbb{B}_R}$ and such that $f(0) \neq 0.$ Then
\begin{equation}\label{eq2}
n_f(r)=P_f(r)+2S_f(r)
 \le \frac{\log M_{f} (R) -\log |f(0)| }{\log R - \log r}
\end{equation}
for any $0 <r <R.$
\end{corollary}
\begin{proof}
First we observe that
\[
\log M_f(R)\ge
\displaystyle\frac{1}{2\pi }\displaystyle \int_0^{2\pi} \int_{\S }\log|f(R\cos\theta +R\sin\theta I)|d\mu(I)d\theta.
\]
Therefore the Jensen's inequality \eqref{jensenNreg} implies:
\begin{align*}
\log M_f(R) - \log|f(0)| &\ge
\sum_{|p_k|<R} m_k\log \frac{R}{|p_k|}\\
&= \left( \sum_{|p_k|<R} m_k\log R \right)
- \sum_{|p_k|<R} m_k\log|p_k|\\
&= n_f(R)\log R - \sum_{|p_k|<R} m_k\log|p_k|\\
&= n_f(R)\log R - \sum_{|p_k|\le r} m_k\log|p_k|
   -\sum_{r<|p_k|<R} m_k\log|p_k|\\
&\ge n_f(R)\log R - \sum_{|p_k|\le r} m_k\log r
   -\sum_{r<|p_k|<R} m_k\log R\\
&= n_f(R)\log R - n_f(r)\log r
   -(n_f(R)-n_f(r))\log R\\
&= n_f(r)\left( \log R - \log r \right).\\
\end{align*}
Thus 
\begin{equation}\label{eq42}
\log M_f(R) - \log|f(0)| \ge  n_f(r)\left( \log R - \log r \right)
\end{equation}
for all $0<r<R$ such that $f$ has no zeros on $\partial\B_R$.
By continuity (in $R$) it follows that \eqref{eq42} holds
without any assumption whether there are zeros on $\partial\B_R$
or not.

This yields the assertion.
\end{proof}

\begin{corollary}
Let $f:\B_1\to\B_1$ be a regular function with $f(0)\ne 0$.

Then there is no zero of $f$ in $\B_r$ for any $r<|f(0)|$.
\end{corollary}
\begin{proof}
Assume the contrary. Then $n_f(r)\ge 1$ while
\[
\lim_{R\to 1} \frac{\log M_{f} (R) -\log |f(0)| }{\log R - \log r}
\le \frac{-\log|f(0)|}{-\log r}<1
\]
leading to a contradiction.
\end{proof}
The interested reader can find in \cite{AB} and in \cite{perotti} other results about Jensen's Formula but in a slightly different context.

\section*{Acknowledgements}
The two authors were partially supported by GNSAGA of INdAM and by INdAM itself.
C. Bisi was also partially supported by PRIN \textit{Variet\'a reali e complesse:
geometria, topologia e analisi armonica}. The two authors would like also to thank the anonymous referees for their insightful comments.

\end{document}